\documentclass[reqno]{amsart}
\usepackage{epsfig}
\usepackage{amsmath, amsfonts }
\usepackage{amssymb}
\usepackage{amscd}
\usepackage{eufrak}
\usepackage{amssymb}
\usepackage{latexsym}
\usepackage[all]{xy}
\author{Ndouné Ndouné  }
\title{Cluster algebras Arising from the infinity-gon}

\newtheorem{theo}{Theorem}[section]
\newtheorem{prop}{Proposition}[section]
\newtheorem{definition}{Definition}[section]
\newtheorem{expl}{Example}[section]

\newtheorem{lem}{Lemma}[section]
\newtheorem{cor}{Corollary}[section]
\newtheorem{rem}{Remark}[section]
\usepackage{graphicx}
\usepackage{t1enc}

\begin{document}
\maketitle

\begin{abstract}
We introduce a handy construction of cluster algebras of type $\mathbb{A}_{\infty}$,
we give a complete classification of the cluster algebras arising from the infinity-gon, and finally we construct the category of the diagonals
of the infinity-gon and show that it is triangle-equivalent to the infinite cluster category of type
$\mathbb{A}_{\infty}$ described by Holm and J{\o}rgensen.

\end{abstract}

\maketitle

\section{Introduction}
Since their invention by Fomin and Zelevinsky in [\ref{fz}] cluster algebras have found applications in various areas of
 mathematics like Lie theory, quiver representations, Calabi-Yau algebras, Teichmüller theory, tropical geometry, Poisson
geometry, integrable systems, combinatorics and  mathematical physics see [\ref{pal}], [\ref{dup}], [\ref{geshva}],
 [\ref{fogo}].

              A cluster algebra is a commutative ring with a distinguished set of generators, called $cluster$ $variables$. The set of all cluster variables
 is constructed recursively from an initial set using a procedure
called $mutation$. Generators are organized into clusters and each cluster contains
exactly $n$ clusters variables. The study of cluster structures for 2-Calabi-Yau categories in [\ref{birs}] led however to the
introduction of cluster algebras with countable clusters.
 The infinite cluster category $\mathcal{D}$ of Dynkin type $\mathbb{A}_{\infty}$ was
constructed by P. Jorgensen in [\ref{pj}].
The cluster tilting subcategories of $\mathcal{D}$ were classified by using the triangulations of infinity-gon in [\ref{holjor}].
The Caldero-Chapoton map has been introduced in [\ref{cc}] and [\ref{ck}] by Caldero, Chapoton and Keller to formalize the
connection between the Fomin-Zelevinsky cluster algebras and the cluster category of Buan, Marsh, Reineke, Reiten and Todorov. The analogue of this map
was constructed between infinite cluster algebras of type $\mathbb{A}_{\infty}$ and the infinite cluster category $\mathcal{D}$ by
Jorgensen and Palu, see [\ref{jorpal}].

This paper is devoted to the study of cluster algebras of type $\mathbb{A}_{\infty}$, cluster algebras arising from the  triangulations and the categorification of the infinity-gon.
In this paper we first give a handy construction of each cluster algebra $\mathcal{B}$ of type $\mathbb{A}_{\infty}.$
More specifically we prove that the cluster algebra $\mathcal{B}$ is a particular subalgebra of the projective limit algebra $\mathcal{A}$ of a particular projective system $(\mathcal{A}^{i},p_{i,j})_{i,j\geq 1}$, where $\mathcal{A}^{n}$ is a cluster algebra of type $\mathbb{A}_n$, and $p_{i,j}:\mathcal{A}^{j}\longrightarrow \mathcal{A}^{i}$ is a surjective cluster morphism in the sense of Assem, Dupont and Schiffler.

\begin{theo}
The cluster algebra $\mathcal{B}$ is the proper $\mathbb{Z}$-subalgebra of the algebra $\mathcal{A}=\underset{\longleftarrow}{\lim}\mathcal{A}^{n}$, consisting of the
ultimately constant elements of $\mathcal{A}$.
\end{theo}

We next continue the study of cluster algebras arising from the infinity-gon started by Grabowski and Gratz in [\ref{gs}]. In this study, we associate to each triangulation $T$ of the infinity-gon, the cluster algebra
$\mathcal{A}(T)$ as done by Fomin, Shapiro and Thurston in [\ref{fst}] for the case of the marked surfaces with a finite marked points. For the case of marked surfaces with a finite number of marked points, Fomin, Shapiro and Thurston have shown that the cluster algebra associated to a triangulated surface does not depend upon the choice of triangulation. We shall show that this result does not hold for cluster algebras arising from the triangulations of the infinity-gon. However, we give a complete classification of the clusters algebras arising from the infinity-gon using the notion of congruence between triangulations inside the set of all triangulations. In [\ref{fz1}] Fomin and Zelevinsky have considered the notion
of $strong$ $isomorphisms$, by which they mean an isomorphism of the cluster algebras which
maps cluster to cluster.

\begin{theo}
Let $T$ and $T'$ be two triangulations of the infinity-gon, and $\mathcal{A}(T)$ and $\mathcal{A}(T')$ the associated cluster algebras.
Then $T$ and $T'$ are congruent if and only if the clusters algebras $\mathcal{A}(T)$ and $\mathcal{A}(T')$ are strongly isomorphic.
\end{theo}

We also define the category of diagonals of the infinity-gon as the one of the ($n$+3)-gon constructed by Caldero, Chapoton and Schiffler in [\ref{ccs}]. It is well-known, that the category of diagonals of the ($n$+3)-gon is equivalent to the cluster category of Buan, Marsh, Reineke, Reiten, and Todorov [\ref{bmrrt}] for the quiver of type $\mathbb{A}_{n}$. Here we show that the category of diagonals of the infinity-gon $\mathcal{C}$ is equivalent to the infinite cluster category $\mathcal{D}$ of type $\mathbb{A}_{\infty}$  of J{\o}rgensen [\ref{pj}]. Note that the category $\mathcal{C}$ appears implicitly in [\ref{pj}].

\begin{theo}
The categories $\mathcal{C}$ and $\mathcal{D}$ are triangle-equivalent.
\end{theo}
As a consequence of this result, we give a description of the Auslander-Reiten triangles in geometric terms inspired by [\ref{bruzh}].
 Our paper is
organized as follows.

        In section 2, we introduce a special projective system of clusters algebras of type $\mathbb{A}_n$ and give the relation between the projective limit
of this system and the corresponding cluster algebra of type $\mathbb{A}_{\infty}$; this gives rise to a handy construction of clusters algebras of type $\mathbb{A}_{\infty}$. This construction is handy because
the algebra $\mathcal{B}$ is expressed as a classical sub-algebra of a finite product of $\mathbb{Z}$-algebras.

        In section 3, we give a complete classification of the cluster algebras arising from the infinity-gon.

         Finally, in section 4, we construct the category $\mathcal{C}$ of diagonals of the infinity-gon and show that it is triangle-equivalent to
the category $\mathcal{D}$ . Therefore, inspired by [\ref{bruzh}], and we give a description of Auslander-Reiten triangles of $\mathcal{C}$ using the diagonals of the infinity-gon.

\begin{center}ACKNOWLEDGMENTS
\end{center}
        I would like to thank my supervisors Ibrahim Assem and Vasilisa Shramchenko  for their patience and availability.
I thank Grégoire Dupont for interesting and helpful discussions, David Smith for valuable remarks. I also would like to thank Yann Palu and Peter J{\o}rgensen
for some clarifications of the infinite cluster category of type $\mathbb{A}_{\infty}.$

\section{Cluster algebras of type $\mathbb{A}_{\infty}$}

\subsection{Basic construction}

We recall that a $quiver$ is a quadruple $Q=(Q_0,Q_1,s,t)$ consisting of two sets,
$Q_0$(whose elements are called $points$) and $Q_1$(whose elements are called $arrows$)
and two functions $s,t:Q_1\longrightarrow Q_0$ associating to each arrow $\alpha\in Q_1$
its so-called $source$ $s(\alpha)$ and $target$ $t(\alpha).$ If $i=s(\alpha)$ and $j=t(\alpha),$
we denote this situation by $i\overset{\alpha}{\longrightarrow}
 j$. Given a point $i$, we set  $i^{+} = \{\alpha\in Q_1| s(\alpha)=i\}$ and
 $i^{-}=\{\alpha\in Q_1| t(\alpha)=i\}$. We say that a quiver $Q$ is $locally$ $finite$ if
 for each $i\in Q_0$, the sets $i^{+}$ and $i^{-}$ are finite.

             Let $Q$ be a countably infinite, but locally finite quiver without cycles of length at most two,
  and let
   $X = \{x_{n}| n\geq 1\}$ be a countable set of undeterminates.  where we agree that the point
$i$ of the quiver $Q$ corresponds to the variable $x_i$. We define the mutation $\mu_k$ in $k\in Q_0$ exactly as in the case of a finite
quiver, that is $\mu_{k}(X,Q)=(X^{'},Q^{'})$, where $Q^{'}$ is the quiver obtained from $Q$
by performing the following operations:\\
- for any path of $i\longrightarrow j\longrightarrow k$ of length two having $k$ as midpoint,
we insert a new arrow $i\longrightarrow j.$\\
- all arrows incident to the point $k$ are reversed, \\
- all newly occurring cycles of length two are deleted.\\
Clearly, $Q^{'}$ is still locally finite.

             On the other hand, $X^{'}$ is a countable set of variables defined as follows:\\
$X^{'}=(X\setminus \{x_k\})\cup \{x_{k}^{'}\}$ where $x_{k}^{'}$ is obtained from $X$ by the
so-called $exchange$ $relation$ $$x_{k}x_{k}^{'}=\underset{\alpha\in i^{+}} \prod x_{t(\alpha)}+
\underset{\alpha\in i^{-}} \prod x_{s(\alpha)}.$$  These operations are performed inside the
field $\mathcal{F}=\mathbb{Q}(X)$ of rational functions over the undeterminates $x_{n}$, called the $ambient$ $field$.
One verifies exactly as in the case of a finite quiver that $\mu_{k}^{2}(X,Q)=(X,Q).$

      From now on, let $Q$ be a quiver having as underlying graph the infinite tree $\mathbb{A}_{\infty}$

    $$
\xymatrix @=20pt
{
&1\ar@{-}[r]
&2\ar@{-}[r]
&3\ar@{-}[r]
&...\ar@{-}[r]
&n-1\ar@{-}[r]
& n\ar@{-}[r]
&...
}
$$
and let $X = \{x_{n}| n\in\Bbb{N}\}$ be a countable set of undeterminates.
          Each pair $(X^{'},Q^{'})$ obtained from $(X,Q)$ by a finite sequence of mutations is
called a $seed$, and the set $X^{'}$ is called a $cluster$. The elements of $X^{'}$ are called $cluster$
$variables$. The pair $(X,Q)$ is called the $initial$ $seed$, and $X$ is called the $initial$ $cluster$.
\begin{definition}
The cluster algebra $\mathcal{B}=\mathcal{A}(X,Q)$ is the $\mathbb{Z}$-subalgebra of $\mathcal{F}$
generated by the set $\mathcal{X}$ which is the union of all possible sets of variables obtained from
$X$
by finite sequences of mutations.\end{definition}

            Gekhtman, Shapiro and Vainshtein showed in [\ref{geshva}] that for every seed $(\tilde{X},\tilde{Q})$
of a given cluster algebra, the quiver $\tilde{Q}$ is uniquely defined by the cluster $\tilde{X}$. Because mutation is a local
operation, this result remains true for the countable seeds above.
Our objective in this first section is to give a handy construction of the cluster algebra
$\mathcal{B}=\mathcal{A}(X,Q)$.

\subsection{A projective system of cluster algebras}

In this subsection we denote by $\overrightarrow{\mathbb{A}}_{n}$
the linearly oriented quiver of type $\mathbb{A}_{n}$, $1\longrightarrow 2\longrightarrow 3\longrightarrow
\cdot \cdot \cdot \longrightarrow n$
and by $X^{n} = \{x_{1},x_{2},..., x_{n}\}$ an associated set of variables. Let
$\mathcal{F}^{n}=\mathbb{Q}(x_{1},x_{2}, ... , x_{n})$ be the field of rational functions on the $x_{i}$ for $1\leq i\leq n$ (with rational coefficients)
and $\mathcal{X}^{n}$ be the union of all possible sets
of variables obtained from $X^{n}$ by successive mutations. This data defines a cluster algebra
$\mathcal{A}^{n}=\mathcal{A}(X^{n},\overrightarrow{\mathbb{A}}_{n})$ having $(X^{n},\overrightarrow{\mathbb{A}}_{n})$ as initial seed. We recall that
the Laurent phenomenon asserts that each cluster variable in $\mathcal{A}^{n}$ can be expressed as a
Laurent polynomial in the $x_i$, with $1\leq i\leq n$, that is, such a variable is of the form

\begin{displaymath} \frac{p(x_{1},x_{2}, ... , x_{n})}{\prod\limits_{l=1}^{n} x_{l}^{d_l}}
 \end{displaymath}
 where $p\in \mathbb{Z}[x_{1},x_{2}, ..., x_{n}]$ and $d_{l}\geq 0$ for all $l$, with $1\leq l \leq n$. The $positivity$ $theorem$ asserts that
 all coefficients of the polynomial $P$ are non-negative integers. The positivity theorem holds because the cluster algebra $\mathcal{A}^{n}$ is of type
 $\mathbb{A}_{n}$.

        Now let $i,j$ be positive integers with $i\leq j,$ we define the map
 $p_{i,j}:$ $\mathcal{A}^{j}\longrightarrow \mathcal{A}^{i}$
 on the generators of $\mathcal{A}^{j}$ as follows

\begin{displaymath} p_{i,j}\left(\frac{p(x_{1},x_{2}, ... , x_{i},x_{i+1},...,x_{j})}
{\prod\limits_{l=1}^{j} x_{l}^{d_l}}\right)
=\frac{p(x_{1},x_{2}, ... , x_{i},1,...,1)}{\prod\limits_{l=1}^{i} x_{l}^{d_l}}
.\end{displaymath}
Since $p_{i,j}$ is an evaluation, it is a morphism of $\mathbb{Z}$-algebras.
Clearly, we have $p_{i,i}=id_{\mathcal{A}^{i}}$ and if $i< j$, $p_{i,j}=p_{i,i+1}p_{i+1,i+2}...p_{j-1,j}.$
Thus, if $i\leq j\leq k$, then $p_{i,j}p_{j,k}=p_{i,k}.$ Our first objective is to prove that, if $i\leq j$,
then $p_{i,j}$ is actually a surjective morphism from $\mathcal{A}^{j}$ to $\mathcal{A}^{i}.$

\begin{prop} With the above notation, $\mathcal{A}^{i}=\mathbf{Z}[p_{i,j}(\mathcal{X}^{j})].$ In particular,
$p_{i,j}$ is a surjective morphism of $\mathbb{Z}$-algebras from $\mathcal{A}^{j}$ to $\mathcal{A}^{i}.$

\end{prop}

\begin{proof}
Because of the above equalities, it suffices to show that, for each $n\geq 2,$ we have
$\mathcal{A}^{n-1}=\mathbb{Z}[p_{n-1,n}(\mathcal{X}^{n})].$ This is done by induction on $n$.

           Assume first that $n=2.$ In this case $\mathcal{A}^{2}=\mathbb{Z}[x_{1},x_{2},\frac{1+x_{2}}{x_{1}},\frac{1+
{x_1}+x_{2}}{x_{1}x_{2}},\frac{1+x_{1}}{x_{2}}]$ while $\mathcal{A}^{1}=\mathbb{Z}[x_{1},\frac{2}{x_{1}}].$
The morphism $p_{1,2}:\mathcal{A}^{2}\longrightarrow \mathcal{A}^{1}$ is defined on the generators as follows:\\
$p_{1,2}(x_{1})=x_{1}$, $p_{1,2}(x_{2})=1$, $p_{1,2}(\frac{1+x_{2}}{x_{1}})=\frac{2}{x_{1}},$
$p_{1,2}(\frac{1+x_{1}+x_{2}}{x_{1}})=1+\frac{2}{x_{1}},$ and $p_{1,2}(\frac{1+x_{1}}{x_{2}})=1+x_{1}.$
Thus, clearly,  $\mathcal{A}^{1}=\mathbb{Z}[p_{1,2}(\mathcal{X}^{2})]$ so that $p_{1,2}:\mathcal{A}^{2}
\longrightarrow \mathcal{A}^{1}$ is a surjective morphism of $\mathbb{Z}$-algebras.

       We now assume that, for every $j<n,$ we have  $\mathcal{A}^{j-1}=\mathbb{Z}[p_{j-1,j}(\mathcal{X}^{j})]$ and
show that $\mathcal{A}^{n-1}=\mathbb{Z}[p_{n-1,n}(\mathcal{X}^{n})].$ For this purpose, we use the
categorification of the cluster algebras $\mathcal{A}^{n},$ and $\mathcal{A}^{n-1},$ as in [\ref{bmrrt}]. The Auslander-Reiten
quiver $\Gamma^{n}$ of the cluster category attached to $\mathcal{A}^{n}$ is of the form

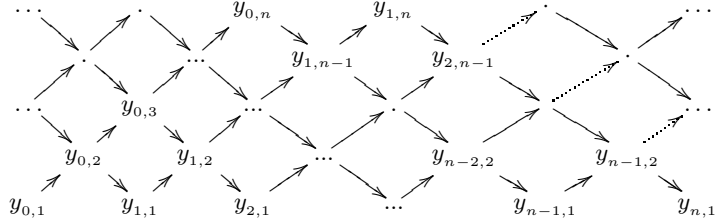
\begin{center}
{\small
\[\xymatrix@R=6pt@C=1pt{
&&&&&&& &&&& \\
&& \ldots \ar[dr] & & . \ar[dr] & & y_{_{0,n}} \ar[dr] & & y_{_{1,n}} \ar[dr] & &
\cdot \ar[dr] & & \ldots && \\
&&& . \ar[dr] \ar[ur] & & ... \ar[dr] \ar[ur] & & y_{_{1,n-1}} \ar[dr] \ar[ur] & & y_{_{2,n-1}} \ar[dr] \ar@{.>}[ur] & &
\cdot \ar[dr] \ar[ur] \\
&& \ldots  \ar[ur] \ar[dr] & & y_{_{0,3}} \ar[ur] \ar[dr] & & ... \ar[ur] \ar[dr] & & . \ar[ur] \ar[dr]& &
\ar@{.>}[ur] \ar[dr] & &  \ldots \\
&&& y_{_{0,2}} \ar[ur] \ar[dr] & & y_{_{1,2}} \ar[ur] \ar[dr] & & ... \ar[ur] \ar[dr] & & y_{_{n-2,2}} \ar[ur] \ar[dr] & &
y_{_{n-1,2}} \ar@{.>}[ur] \ar[dr] \\
&& y_{_{0,1}} \ar[ur] & & y_{_{1,1}} \ar[ur] & & y_{_{2,1}} \ar[ur] & & ... \ar[ur] & & y_{_{n-1,1}} \ar[ur] & & y_{_{n,1}} \\
}\]
}
\end{center}
where we agree to identify each point of $\Gamma^{n}$ with the corresponding cluster variable and $y_{0,i}=x_{i}$
for each $i$ such that $1\leq i\leq n;$ and we denote by $y_{i,j}$ with $0\leq i\leq n$, $1\leq j\leq n$ and $i+j\leq n+1$,
the clusters variables of $\mathcal{A}^{n}.$ Because the quiver $\overrightarrow{\mathbb{A}}_{n}$ is of Dynkin type $\mathbb{A}_{n}$,
the cluster algebra $\mathcal{A}^{n}$ is of finite type and the
Auslander-Reiten quiver $\Gamma^{n}$ lies on a Moebius strip. Thus $y_{0,i}=y_{i,s}$, where $i+s=n+1$ and $1\leq i,s\leq n$.

        We denote by $\Gamma^{n-1}$ the Auslander-Reiten quiver of the cluster category attached to $\mathcal{A}^{n-1}.$ It
 is of the form

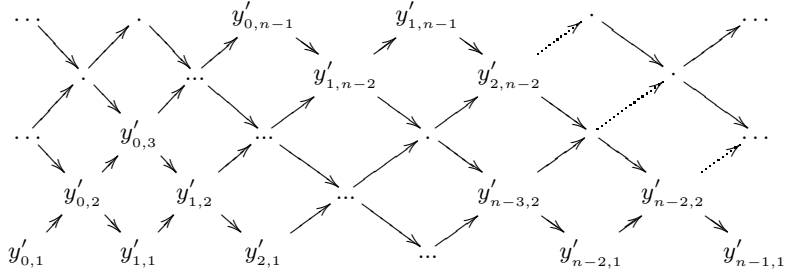
\begin{center}
{\small
\[\xymatrix@R=6pt@C=1pt{
&&&&&&& &&&& \\
&& \ldots \ar[dr] & & . \ar[dr] & & y'_{_{0,n-1}} \ar[dr] & & y'_{_{1,n-1}} \ar[dr] & &
\cdot \ar[dr] & & \ldots && \\
&&& . \ar[dr] \ar[ur] & & ... \ar[dr] \ar[ur] & & y'_{_{1,n-2}} \ar[dr] \ar[ur] & & y'_{_{2,n-2}} \ar[dr] \ar@{.>}[ur] & &
\cdot \ar[dr] \ar[ur] \\
&& \ldots  \ar[ur] \ar[dr] & & y'_{_{0,3}} \ar[ur] \ar[dr] & & ... \ar[ur] \ar[dr] & & . \ar[ur] \ar[dr]& &
\ar@{.>}[ur] \ar[dr] & &  \ldots \\
&&& y'_{_{0,2}} \ar[ur] \ar[dr] & & y'_{_{1,2}} \ar[ur] \ar[dr] & & ... \ar[ur] \ar[dr] & & y'_{_{n-3,2}} \ar[ur] \ar[dr] & &
y'_{_{n-2,2}} \ar@{.>}[ur] \ar[dr] \\
&& y'_{_{0,1}} \ar[ur] & & y'_{_{1,1}} \ar[ur] & & y'_{_{2,1}} \ar[ur] & & ... \ar[ur] & & y'_{_{n-2,1}} \ar[ur] & & y'_{_{n-1,1}} \\
}\]
}

\end{center}
where again each point is identified with the corresponding cluster variable. Therefore $y_{0,i}^{'}=x_{i}$ for all $i$
such that $1\leq i\leq n-1;$ and we denote by $y'_{i,j}$ with $0\leq i\leq n-1$, $1\leq j\leq n-1$ and $i+j\leq n$,  the clusters
variables of $\mathcal{A}^{n-1}.$ Because the quiver $\overrightarrow{\mathbb{A}}_{n-1}$ is of Dynkin type $\mathbb{A}_{n-1}$, the cluster algebra
$\mathcal{A}^{n-1}$ is of finite type and the
Auslander-Reiten quiver $\Gamma^{n-1}$ lies on a Moebius strip. Thus, $y'_{0,i}=y'_{i,s}$,
where $r+s=n$ and $1\leq i,s\leq n$.

       We say that a point $y_{i,j}$ of $\Gamma^{n}$ is $stable$ provided $p_{n-1,n}(y_{i,j})=y'_{i,j}.$ Clearly, all the points
$y_{0,i}$ with $1\leq i\leq n-1$ are stable. It then follows from the definition of mutation that, for every pair
$(i,j)$, such that $i+j\leq n-1$, the point $y_{i,j}$ is stable.\\
Now it remains to consider the points $\{y_{i,j}| n\leq i+j\leq n+1\}$ of $\Gamma^{n}.$ Because $p_{n-1,n}$
is an evaluation, we may write $p_{n-1,n}(y_{i,j})=y_{i,j}(1)$ for brevity.

         The proof is completed in the following three steps $(a)$ $(b)$ $(c)$.

        $(a)$ We first claim that, if $i+j=n$ and $n\geq 1,$ then $y_{i,j}(1)=y_{i,j}^{'}.$ This is done by induction on
$i.$

       If $i=1,$ we have $y_{1,n-1}=\frac{1+y_{1,n-2}y_{0,n}}{y_{0,n-1}}$. The evaluation of $y_{1,n-1}$ at 1 is given by:
\begin{eqnarray*}y_{1,n-1}(1)&=&\frac{1+y_{1,n-2}y_{0,n}}{y_{0,n-1}}(1)\\&=&\frac{1+y_{1,n-2}}
{y_{0,n-1}}\\&=&y'_{1,n-1},\end{eqnarray*}
where we used the stability of $y_{0,n-1}$ and $y_{1,n-2}.$

         Assume now the result valid for all $j\leq i.$ We have
$y_{i+1,n-i-1}=\frac{1+y_{i+1,n-i-2}y_{i,n-i}}{y_{i,n-i-1}}$. The evaluation of $y_{i+1,n-i-1}$ at 1 is given by:
\begin{eqnarray*}y_{i+1,n-i-1}(1)&=&\frac{1+y_{i+1,n-i-2}y_{i,n-i}}{y_{i,n-i-1}}(1)
\\&=&\frac{1+y_{i+1,n-i-2}(1)y_{i,n-i}(1)}{y_{i,n-i-1}(1)}
\\&=&\frac{1+y'_{i+1,n-i-2}y'_{i,n-i}}{y'_{i,n-i-1}}\\&=&y'_{i+1,n-i-1},\end{eqnarray*}
where we used the induction hypothesis and the stability of $y_{i+1,n-i-2}$ and $y_{i+1,n-i-1}.$ This establishes our
claim for the step $(a)$.

      $(b)$ Next we consider the particular case of the variable $y_{1,n}.$ Because $y_{1,n}=\frac{1+y_{1,n-1}}{y_{0,n}}$,
 we have
\begin{eqnarray*}y_{1,n}(1)&=&\frac{1+y_{1,n}(1)}{y_{0,n}(1)}
\\&=&1+y'_{1,n-1}\end{eqnarray*} using point $(a).$\\
$(c)$ Finally, we prove that, if $i+j=n+1$ and $i\geq 2,$ then $p_{n-1,n}(y_{i,j})=y'_{0,i}+y'_{i,n-i}.$ Assume
first $i=2,$ then $y_{2,n-1}y_{1,n-1}=1+y_{1,n}y_{2,n-2}$ yields
\begin{eqnarray*}y_{2,n-1}(1)&=&\frac
{1+y_{1,n}(1)y_{2,n-2}(1)}{y_{1,n-1}(1)}
\\&=&\frac{1+(1+y'_{1,n-1})y'_{2,n-2}}{y'_{1,n-1}}\\&=&\frac{1+y'_{2,n-2}}{y'_{1,n-1}}+y'_{2,n-2}
\\&=&y'_{0,1}+y'_{2,n-2}\end{eqnarray*}
where we use points $(a)$ and $(b).$ Finally, if $i\geq 3,$ then
\begin{eqnarray*}y_{i+1,n-i}(1)&=&\frac{1+y_{i+1,n-i-1}(1)y_{i,n-i+1}(1)}{y_{i,n-i}(1)}\\&=&\frac
{1+y'_{i+1,n-i-1}(y'_{i,n-i+1}+y'_{i,n-i})}{y'_{i,n-i}}
\\&=&\frac{1+y'_{i+1,n-i-1}y'_{i,n-i+1}}{y'_{i,n-i}}+y'_{i+1,n-i-1}\\&=&y'_{0,i+1}+y'_{i+1,n-i-1}\end{eqnarray*}
This completes the proof of our claim.

Finally, it follows easily from $(a)$, $(b)$ and $(c)$ that $\mathcal{A}^{n-1}=\mathbb{Z}[p_{n-1,n}(\mathcal{X}^{n})],$ as
asserted.

\end{proof}

      We want to show that $p_{i,j}$ is a morphism of cluster algebras in the sense of Assem, Dupont and Schiffler in [\ref{asdush}].
In order to define cluster morphisms, we recall the definitions of rooted cluster algebras and rooted cluster morphisms due to
I. Assem, G. Dupont and R. Schiffler.

\begin{definition}
A seed is a triple $\Sigma=(X,ex,B)$ such that:\\
(1) $X$ is a countable set of undeterminates over $\mathbb{Z},$ called the clusters of $\Sigma;$\\
(2) $ex\subset X$ is a subset of $X$ whose elements are the exchangeable variables of $\Sigma;$\\
(3) $B=(b_{x,y})_{x,y\in X}\in M_{X}(\mathbb{Z})$ is a (locally finite) skew-symmetrisable matrix called the exchange matrix of $\Sigma.$

\end{definition}

The elements of $X\backslash ex$ are called the $frozen$ $variables$.
Note that in the above definition, the matrix $B$ can be replaced by a (locally finite) quiver without loops and 2-cycles.

    Let $\Sigma=(X,ex,Q)$ be a seed. We say that $(x_1,x_2,...,x_l)$ is $\Sigma$-$admissible$ if $x_1$ is exchangeable in $\Sigma$ and if, for
every $i\geq 2,$ the variable $x_i$ is exchangeable in $\mu_{x_{i-1}}\circ...\circ \mu_{x_1}(\Sigma)$. The mutations are made along finite admissible
sequences of variables.

    A rooted cluster algebra is defined similarly as the Fomin-Zelevinsky cluster algebras, but the definition of rooted cluster algebras authorises seeds whose clusters are empty.
 Such seeds are called empty seeds and by convention the rooted cluster algebra corresponding to an empty seed is $\mathbb{Z}.$ The rooted cluster algebra is always viewed with its initial seed.
To know more about the different points of view between Fomin-Zelevinsky cluster algebras and rooted cluster algebras, we refer to [\ref{asdush}, Remark 1.7].

    Let $\Sigma=(X,ex,Q)$ and $\Sigma'=(X',ex',Q')$ be two seeds and let $f:\mathcal{A}(\Sigma)\longrightarrow \mathcal{A}(\Sigma')$ be a map. A sequence
$(x_1,x_2,...,x_l)\subset \mathcal{A}(\Sigma)$ is $(f,\Sigma,\Sigma')$-$biadmissible$ if it is $\Sigma$-admissible and $(f(x_1),...,f(x_l))$ is $\Sigma'$-admissible.
 The following definition is due to Assem, Dupont and Schiffler.
 \begin{definition}
A rooted cluster morphism from $\mathcal{A}(\Sigma)$ to $\mathcal{A}(\Sigma')$ is a ring homomorphism from $\mathcal{A}(\Sigma)$ to $\mathcal{A}(\Sigma')$ such that:\\
(CM1) $f(X)\subset X'\cup \mathbb{Z}$\\
(CM2) $f(ex)\subset ex'\cup \mathbb{Z}$\\
(CM3) For every $(f,\Sigma,\Sigma')$-biadmissible sequence $(x_1,x_2,...,x_l)$, we have $\mu_{x_{l}}\circ...\circ \mu_{x_1,\Sigma}(y)=\mu_{f(x_{l})}\circ...\circ \mu_{f(x_1),\Sigma'}(f(y)).$

\end{definition}

The rooted cluster algebras and the rooted cluster morphisms forms a category, see [\ref{asdush}].

       Let $\Sigma^{n}=(X^{n},ex^{n},Q^{n})$, with $ex^{n}=X^{n}$, then $\Sigma^{n}$ is a seed of the rooted cluster algebra
$\mathcal{A}(\Sigma^{n})$; the cluster algebra $\mathcal{A}(\Sigma^{n})$ coincides with the cluster algebra $\mathcal{A}^{n}$. We also have $\Sigma^{n-1}=\Sigma^{n}\backslash \{x_n\}$,
where the seed $\Sigma^{n}\backslash \{x_n\}$ is defined in [\ref{asdush}, Section 6.2].

\begin{cor}
Each $\mathbb{Z}$-morphism $p_{i,j}$ is a surjective rooted cluster morphism.
\end{cor}

\begin{proof}
By the Proposition 2.1, the map $p_{n-1,n}$ is a $\mathbb{Z}$-morphism induced by the specialisation of $x_n$ to 1. Because of [\ref{asdush}, Proposition 6.10], the
$\mathbb{Z}$-morphism $p_{n-1,n}$ is a surjective rooted cluster morphism. Since $p_{i,j}=p_{i,i+1}\circ p_{i+1,i+2}\circ...\circ p_{j-1,j}$ and each $p_{n-1,n}$ is a surjective
rooted cluster morphism, by [\ref{asdush}, Proposition 2.5]  $p_{i,j}$ is also a surjective rooted cluster morphism.

\end{proof}

\begin{cor}
The family $(\mathcal{A}^{i},p_{i,j})_{i,j\geq 1}$ forms a projective system of cluster algebras.
\end{cor}

         We denote by $\mathcal{A}= \underset{\longleftarrow}{\lim}\mathcal{A}^{n}$ the corresponding projective limit
 in the category of $\mathbb{Z}$-algebras, thus
 $$\mathcal{A} = \{(a_{n})_{n\geq 1}\in\ \underset{n\geq 1}{\Pi}
\mathcal{A}^{n}| p_{i,j}(a_j)=a_i\}.$$ We also denote by $p_{i}:\mathcal{A}\longrightarrow \mathcal{A}^{i}$ the
canonical morphisms induced by the projective limit. Let $a_l$ be a cluster variable of $\mathcal{A}^{i}$, then the element $(a_1,a_2,...,a_l,a_l,...)$ with $a_j=a_l$ for $j\geq l\geq i$,
is an element of $\mathcal{A}$ and $p_i(a_1,a_2,...,a_l,a_l,...)=a_i$; therefore $p_i$ is a surjective homomorphism of $\mathbb{Z}$-algebras. The $p_{i}$ are called the canonical projections morphisms.
 We thus have a commutative diagram
\[
\xymatrix{
\ldots \ar@{->>}[r] & \mathcal{A}^i \ar@{->>}[r]^-{p_{_{i-1,i}}} & \ldots \ar@{->>}[r]&  \mathcal{A}^2 \ar@{->>}[r]^-{p_{_{1,2}}} & \mathcal{A}^1 \\
                    & \mathcal{A} \ar@{->>}[u]^-{p_i} \ar@{->>}[urr]^-{p_2} \ar@{->>}[urrr]_{p_1}
}
\]

It was shown that the category of rooted cluster algebras admits countable coproducts  [\ref{asdush}, Lemma 5.1] and
does not generally admit products [\ref{asdush}, Proposition 5.4].
\begin{definition}
An element $(a_{n})_{n\geq 1}$ of  $\mathcal{A}=\underset{\longleftarrow}{\lim}\mathcal{A}^{n}$ is ultimately constant
if there exists $j\in \mathbb{N}$ such that
$a_n=a_j$
for all $n\geq j$.
\end{definition}

Let $Q$ be a linearly oriented quiver of type $\mathbb{A}_{\infty}$ having 1 as unique source. Here $(X,Q)$ is a seed of the cluster algebra $\mathcal{B}$.
We want to understand the relation between the cluster algebra $\mathcal{B}$ and the $\mathbb{Z}$-algebras $\mathcal{A}.$ Because the algebra $\mathcal{A}$
is a projective limit of cluster algebras of finite type, it allows to express the cluster algebra $\mathcal{B}$ in term of $\mathcal{A}.$
Our first result is the following.
\begin{theo}
The cluster algebra $\mathcal{B}$ is the proper $\mathbb{Z}$-subalgebra of the algebra $\mathcal{A}$, consisting of the
ultimately constant elements of $\mathcal{A}=\underset{\longleftarrow}{\lim}\mathcal{A}^{n}$.
\end{theo}

\begin{proof}
Assume that $Q$ is the quiver having as underlying graph the infinite tree $\mathbb{A}_{\infty}$
with linear orientation and for unique source the vertex 1. Recall that $\mathcal{B}=\mathcal{A}(X,Q)$
is a cluster algebra of seed $(X,Q)$, where  $X = \{x_{n}, n\geq 1\}$. Let $Y = \{y_{n}, n\geq 1\}$ be a new set of undeterminates whose
elements are defined by $y_1=(x_{1},x_{1},x_{1},...)$,
 $y_2=(1,x_{2},x_{2},x_{2},...)$
$y_3=(1,1,x_{3},x_{3},x_{3},...)$,..., $y_i=(1,1,1,...,1,x_{i},x_{i},x_{i},...),$....

      By definition, all  $y_i$ are elements of $\mathcal{A}.$
 Let
$\mathcal{\tilde{F}}=\mathbb{Q}(Y)$ be the field of rational functions over $y_{i}$(with rational coefficients),
we call $\mathcal{\tilde{F}}$ the $ambient$ $field$.

        The cluster algebra $\mathcal{\tilde{A}}=\mathcal{A}(Y,Q)$ is the $\mathbb{Z}$-subalgebra of $\mathcal{\tilde{F}}$
generated by the set $\mathcal{Y}$ which is the union of all possible sets of variables obtained from
$Y$
by successive mutations. We define the map $\varphi:\mathcal{B}\longrightarrow \mathcal{\tilde{A}}$ by setting
$\varphi(x_i)=y_i$ and we extend it to all cluster variables of $\mathcal{B}$ by respecting mutations, that is if $z_i=\mu_{x_{i_k}}...\mu_{x_{i_2}}\mu_{x_{i_1}}(x_i)$
then $\varphi(z_i)=\mu_{\varphi(x_{i_k})}...\mu_{\varphi(x_{i_2})}\mu_{\varphi(x_{i_1})}(\varphi(x_i))$. We extend again $\varphi$ to
an injective morphism of $\mathbb{Z}$-algebras. Thus $\varphi$ is a monomorphism of $\mathbb{Z}$-algebras.

              Let $x=p(x_{1},x_{2},...,x_{k})$ be an element of $\mathcal{X}$ then
$y=p(y_{1},y_{2},...,y_{k})$ is an element of $\mathcal{Y}.$ By definition we have $\varphi(x)=y$.

This shows that the morphism $\varphi$
is an isomorphism between $\mathcal{B}$ and  $\varphi(\mathcal{B})=\mathcal{\tilde{A}}.$

Each ultimately constant element of
$\mathcal{A}$ belongs
to $\mathcal{\tilde{A}}.$ But the element $a=(x_{1},x_{1}x_{2},x_{1}x_{2}x_{3},...,x_{1}x_{2}x_{3}...x_{k},...)$ is an element
 of $\mathcal{A}$
which does not belong to $\mathcal{\tilde{A}}$. Therefore $\mathcal{\tilde{A}}$ is a proper $\mathbb{Z}$-subalgebra of the algebra $\mathcal{A}.$

            Now let $\underline{x} = \{q_{n}|n\geq 1\}$ be a cluster of $\mathcal{B},$ by the definition of $\varphi$ we have
$\varphi(\mu_{q_{k}}(\underline{x}))= \mu_{\varphi(q_{k})}(\varphi(\underline{x}))$ and $\varphi(\underline{x})$ is a cluster
of $\tilde{\mathcal{A}}.$
It follows that $\varphi$ is a cluster isomorphism. For more details about cluster isomorphisms, we refer to [\ref{asdush}].
This completes the proof.
\end{proof}

 \begin{rem}
Let $Q$ be a quiver of type $\mathbb{A}_{\infty}$ and  $Q^{n}$ the full sub-quiver of $Q$ whose set of vertices is $Q^{n}_{0}=\{1,2,...,n\}$.
We denote by $\mathcal{B'}$ the cluster algebra of seed $(X,Q)$ and $\mathcal{A'}^{n}$ the cluster algebra of seed $(X^{n},Q^{n}).$ We reproduce the above construction
with $\mathcal{B'}$ playing the role of $\mathcal{B}$ and $\mathcal{A'}^{n}$ playing the role of $\mathcal{A}^{n}$. Then the Theorem 2.1 remains true for any cluster algebra of type $\mathbb{A}_{\infty}.$
\end{rem}

\begin{cor}
The Laurent phenomenon and the positivity theorem hold for the cluster algebra $\mathcal{B}$.
\end{cor}

\begin{proof}
Let $x$ be a cluster variable of the cluster algebra $\mathcal{B}$. Then there exists a non-negative integer $n$ such that $x$ is identified to a cluster variable $a_m$ of $\mathcal{A}^{n}.$ We have $x=(a_1,a_2,...,a_m,a_m,...)$ with $m\leq n$. Thus by Theorem 2.1, $a_m=q(x_1,x_2,...,x_n)$ if and only if $x=q(y_1,y_2,...,y_n)$, where $Q$ is a Laurent polynomial.
Since $a_m$ is a cluster variable of $\mathcal{A}^{n},$ and it is well-known that the Laurent phenomenon and the positivity theorem hold for the cluster algebra of type $\mathbb{A}_{n}$, then
the Laurent phenomenon and the positivity theorem hold for the cluster algebra  $\mathcal{B'}.$

\end{proof}

\section{Cluster algebras arising  from infinity-gon}
Fomin, Shapiro and Thurston initiated a study of the cluster algebras arising from triangulations of a surface with
boundary and finitely many marked points in [\ref{fst}]. In this approach, it was shown that the cluster algebra associated to a triangulation
of a marked surface $(S,M)$ depends only on the surface $(S,M)$ and not on the choice of triangulation. As we shall see
this is not true in the case of the infinity-gon. Our objective in this section is to classify the cluster algebras arising
from the infinity-gon.

\subsection{Triangulations of the infinity-gon}
In this subsection, we classify the triangulations of the infinity-gon $\mathbb{S}$ using the notions of connected component and
frozen arc which will be defined later.

            We adopt the same philosophy as that of [\ref{holjor}], that is, we view the integers as the vertices of the infinity-gon
and the pairs of integers as the arcs.

            Let $(m,n)$ be an arc of the infinity-gon, with $m< n$. If $n-m=1,$ we
say that the arc $(m,n)$ is a $boundary$ arc, and if $n-m\geq 2,$ we say that $(m,n)$ is a $internal$ $arc$ of the infinity-gon. The
In the following internal arcs will simply called arcs. Two arcs $(m,n)$ and $(p,q)$
are said to $cross$ if we have either $m< p< n< q$  or $p< m< q< n.$ A $triangulation$
of $\mathbb{S}$ is a maximal set of non-crossing arcs.

\begin{definition}
A triangulation $T$ of $\mathbb{S}$ is called a $zigzag$ $triangulation$ if it is of the form $T= \{(-n+n_{0},n+n_{0}),(-n+n_{0}-1,n+n_{0})/ n\geq 1\}$
or $T= \{(-n+n_{0},n+n_{0}),(-n+n_{0},n+n_{0}+1)/ n\geq 1\},$ where $n_0$ is a given integer.
\end{definition}

\begin{expl}

 Let $T$ and $T'$ be the sets of arcs defined by  $T= \{(-n,n),(-n,n+1)/ n\geq 1\}$ and $T'= \{(-n,-1),(-1,1),(1,n)/ n\geq 3\}.$
 One can show that $T$ and $T'$ are two triangulations of $\mathbb{S}$ whose illustrations are respectively the following:
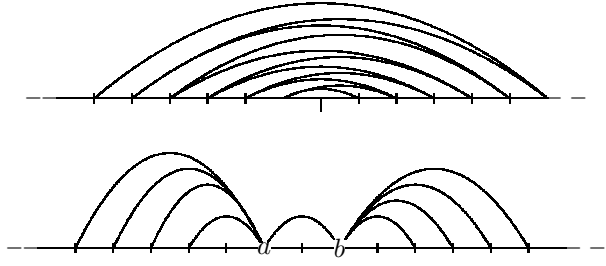
\begin{center}

\[
\xymatrix @-4.0pc @! {
       \rule{0ex}{7.5ex} \ar@{--}[r]
     & *{}\ar@{-}[r]
     & *{\rule{0.1ex}{0.8ex}} \ar@{-}[r] \ar@/^3.0pc/@{-}[rrrrrrrrrrrr]
     & *{\rule{0.1ex}{0.8ex}} \ar@{-}[r]\ar@/^2.50pc/@{-}[rrrrrrrrrrr]
     & *{\rule{0.1ex}{0.8ex}} \ar@{-}[r]\ar@/^2.0pc/@{-}[rrrrrrrrr]
     & *{\rule{0.1ex}{0.8ex}} \ar@{-}[r] \ar@/^1.3pc/@{-}[rrrrrrr]
     & *{\rule{0.1ex}{0.8ex}} \ar@{-}[r] \ar@/^0.8pc/@{-}[rrrrr]
     & *{} \ar@{-}[r] \ar@{-}[r] \ar@/^0.4pc/@{-}[rrr]
     & *{} \ar@{-}[r]\ar@/^-0.4pc/@{-}[llrr]
     & *{\rule{0.1ex}{0.8ex}} \ar@{-}[r] \ar@/^-0.3pc/@{-}[lllr]
     & *{\rule{0.1ex}{0.8ex}} \ar@{-}[r] \ar@/^-0.6pc/@{-}[llll]
     & *{\rule{0.1ex}{0.8ex}} \ar@{-}[r] \ar@/^-1.0pc/@{-}[llllll]
     & *{\rule{0.1ex}{0.8ex}} \ar@{-}[r] \ar@/^-1.5pc/@{-}[llllllll]
     & *{\rule{0.1ex}{0.8ex}} \ar@{-}[r] \ar@/^-2.3pc/@{-}[llllllllll]
     & *{}\ar@{--}[r]
     & *{}
                    }
\]
\end{center}

\begin{center}

\[
\xymatrix @-4.0pc @! {
       \rule{0ex}{7.5ex} \ar@{--}[r]
     & *{}\ar@{-}[r]
     & *{\rule{0.1ex}{0.8ex}} \ar@{-}[r] \ar@/^3.0pc/@{-}[rrrrr]
     & *{\rule{0.1ex}{0.8ex}} \ar@{-}[r]\ar@/^2.50pc/@{-}[rrrr]
     & *{\rule{0.1ex}{0.8ex}} \ar@{-}[r]\ar@/^2.0pc/@{-}[rrr]
     & *{\rule{0.1ex}{0.8ex}} \ar@{-}[r] \ar@/^1.0pc/@{-}[rr]
     & *{\rule{0.1ex}{0.8ex}} \ar@{-}[r]
     & *{a} \ar@{-}[r]
     & *{\rule{0.1ex}{0.8ex}} \ar@{-}[r]
     & *{b} \ar@{-}[r]\ar@/^-1.0pc/@{-}[ll]
     & *{\rule{0.1ex}{0.8ex}} \ar@{-}[r]
     & *{\rule{0.1ex}{0.8ex}} \ar@{-}[r] \ar@/^-1.0pc/@{-}[ll]
     & *{\rule{0.1ex}{0.8ex}} \ar@{-}[r] \ar@/^-1.5pc/@{-}[lll]
     & *{\rule{0.1ex}{0.8ex}} \ar@{-}[r] \ar@/^-2.0pc/@{-}[llll]
     & *{\rule{0.1ex}{0.8ex}} \ar@{-}[r] \ar@/^-2.5pc/@{-}[lllll]
     & *{}\ar@{--}[r]
     & *{}
                    }
\]
.\end{center}
The triangulation $T$ is a $zigzag$ triangulation.
\end{expl}
The following definition is due to Holm and J{\o}rgensen in [\ref{holjor}].
\begin{definition}
Let $T$ be a triangulation of $\mathbb{S}$.\\
$(a)$If for each integer $n$ there are only finitely many arcs in $T$ which are incident to $n$, then $T$ is called $locally$ $finite$.\\
$(b)$If $n$ is an integer such that $T$ contains infinitely many arcs of the form $(m,n)$, then $n$ is called a $left$-$fountain$ of $T$.\\
$(c)$If $n$ is an integer such that $T$ contains infinitely many arcs of the form $(n,p)$, then $n$ is called a $right$-$fountain$ of $T.$\\
$(d)$If $n$ is both a left-fountain and a right-fountain of $T$, then it is called a $fountain$.

\end{definition}

It is shown in [\ref{holjor}]  that if a triangulation of $\mathbb{S}$ has a right-fountain, then it also has a left-fountain and vice versa.
The following result in [\ref{holjor}, Lemma 3.3] characterizes the triangulations of infinity-gon.
\begin{lem}(Holm-J{\o}rgensen)
 Let $T$ be a triangulation of $\mathbb{S}$. Then $T$ has at most
one right-fountain and at most one left-fountain.
\end{lem}

         Now we introduce the notion of connected component of the triangulations of $\mathbb{S}$ before giving a classification of the
triangulations of $\mathbb{S}$. Let $T$ be a triangulation of $\mathbb{S}$ and $\gamma$ an arc of $T.$ We say that the arc $\gamma'$ is the $flip$ of the arc $\gamma$ if the set
$(T\backslash \{\gamma\})\cup\{\gamma'\}$ is a triangulation of $\mathbb{S}.$

         Following J{\o}rgensen and Palu in [\ref{jorpal}], we say that  the arc $\omega=(s,t)$ $spans$ the arc $\delta=(u,v)$ if $s\leq u< v< t$ or $s< u< v\leq t.$ We denote by
         $\mathcal{B}(\omega)$ the set of all arcs spanned by the given arc $\omega.$

\begin{definition}
Let $T$ be a triangulation, and $\tau$ an arc of $T$. We say that an arc $\gamma$ belongs to the connected component of $\tau$ if there exists
a finite sequence of flips $\mu_1,\mu_2,...,\mu_k$ such that $\gamma=\mu_k...\mu_2\mu_1(\tau).$\end{definition}
 We denote by $C_{\tau}$ the connected component of
$\tau.$ The connected components of the arcs of $T$ are called simply the $connected$ $components$ of $T$.

        An arc $\gamma$ of $\mathbb{S}$ is $reachable$ by $T$ if it belongs to a connected components of $T$. If this is not the case then we say that $\gamma$ is
$unreachable$ by $T$. An arc of $T$ which cannot be flipped to any other arc is called a $frozen$ arc.

\begin{lem}
Any triangulation of $\mathbb{S}$ has at most one frozen arc.
\end{lem}

\begin{proof}
Let $T$ be a triangulation of $\mathbb{S}$. Assume that $T$ is locally finite. Let $\gamma$ be an arc of $\mathbb{S}$, then there exists an arc $\zeta$ such that $\gamma$
is an arc of the polygon  $\mathbb{P}_{\zeta}$ bounded by $\zeta$. The
restriction $T_{\zeta}$ of the triangulation $T$ to $\mathbb{P}_{\zeta}$ is a triangulation. Because $\gamma$ is an arc of $\mathbb{P}_{\zeta}$, it is joined by a finite
sequence of flips of arcs of $T_{\zeta}$. Since $T_{\zeta}\subset T$, then $\gamma$ is joined by a finite
sequence of flips of arcs of $T$. Thus, each arc of $\mathbb{S}$ is reachable; hence $T$ has no frozen arc. If $T$ has a left-fountain
$m_0$ and a right-fountain $n_0$ such that $n_0-m_0=1,$ then $T$ has two connected components and no frozen arc. If $T$ has a left-fountain
$m_0$ and a right-fountain $n_0$ such that $n_0-m_0\geq 2,$ then $T$ has three connected components and the arc $(m_0,n_0)$ is a frozen arc. Assume that
$T$ possesses another frozen arc $(m_1,n_1),$ then $(m_1,n_1)$ crosses an infinity of arcs of $T$ incident to $m_0$ or an infinity of arcs of $T$ incident to $n_0$.\\
Assume now that $T$ is a triangulation with two frozen arcs $\omega_1$ and $\omega_2$. The arc $\omega_1$ does not span $\omega_2$ and vice versa, because if not, then one
of the two frozen arcs can be flipped to another arc. Each frozen arc bounds a finite connected component, and then it is finite. Therefore, the triangulation $T$ has
more than one left-fountain or more than one right-fountain. This is a contradiction to Lemma 3.1.

\end{proof}

  Let $T$ be a triangulation of $\mathbb{S}$ with a frozen arc $\tau$, we say that $T$ is of Type $(III)_{k}$ with  $k=|\mathcal{B}(\tau)|,$ where $\mathcal{B}(\tau)$ denotes the number of arcs
spanned by the frozen arc $\omega.$
\begin{definition}
A triangulation $T$ of $\mathbb{S}$ is called \\
a) of Type $(I)$ if it has only one connected component;\\
b) of Type $(II)$ if it has two connected components and no frozen arc.\\
c)of Type $(III)_{k}$ if it has two connected components and one frozen arc or
if it has three connected components and one frozen arc, where $k$ is a nonnegative integer.\\
\end{definition}
Note that $(III)_{k}\neq (III)_{k'}$ if and only if   $k\neq k'$.

           The following result gives a classification of the triangulations of $\mathbb{S}$ which uses the notions of connected components and the frozen arc.
\begin{prop}
 Any triangulation of $\mathbb{S}$ is one of the three types $(I)$, $(II)$, $(III)_{k}$.

\end{prop}

\begin{proof}
Let $T$ be a triangulation of $\mathbb{S}$. If $T$ is locally finite, then every arc of $\mathbb{S}$ can be reached by a sequence of flips of arcs of $T.$ Then $T$ is of type (I). If $T$ has a fountain or a left-fountain $m_0$ and a right-fountain $n_0$ with $n_0-m_0=1$, then $T$ has two connected components, and
any arc of each component is reachable. In this case $T$ has no frozen arc, hence $T$ is of type(II).

        If $T$ has a left-fountain $m_0$ and a right-fountain $n_0$ with $n_0-m_0\geq 2$, then $T$ has three connected components. The arc $(m_0,n_0)$ is an arc of $T.$
Assume that $(m_0,n_0)$ is not in $T$; then one of the arcs $(m_0+1,n_0), (m_0, n_0-1)$ belongs to $T$ and one of the arcs $(m_0-1,n_0), (m_0, n_0+1)$
belongs to $T.$ If the arc $(m_0+1,n_0)$ belongs to $T,$ then $(m_0-1,n_0)\in T$ and  $(m_0-1,n_0)$ crosses an infinite number of arcs of $T$ incident to a left fountain $m_0$.
This is a contradiction because $T$ is a triangulation.

          Similarly, if $(m_0,n_0-1)$ belongs to $T$, we get a contradiction. Hence the arc $(m_0,n_0)$ is a frozen arc, and it
is unique by the lemma 4.1.1. Thus $T$ is of type $(III)_k$.

      Assume now that $T$ has $l$ connected components, where $l\geq 4.$ Then only one of the $l$ components is finite, because if not, $T$ would have more than
one frozen arc, and this is a contradiction to Lemma 3.2. Therefore, the triangulation $T$ has at least three infinite connected components. Each of the
infinite connected components of $T$ contains either a right-fountain or a left-fountain, thus $T$ has more than two fountains, this contradicts Lemma 3.
\end{proof}

Fomin, Shapiro and Thurston in [\ref{fst}] associated to a triangulation of a marked surface
$(S,M)$ a finite quiver without cycles of length at most two. Similarly, we associate to each triangulation of $\mathbb{S}$ an infinite quiver without
cycles of length at most two.

        Let $T$ be a triangulation of the infinity-gon $\mathbb{S}$, we associate to $T$ a quiver $Q_{T}.$ The classification of quivers $Q_{T}$ is given in [\ref{gs}, Theorem 3.11].
We associate to the triangulation $T$ the cluster algebra $\mathcal{A}(T)$ of seed $(X_{T},Q_{T})$ in the same way as the one for a marked surface $(S,M).$

\begin{rem}
Proposition 3.1 is equivalent to the Theorem 3.11 of [\ref{gs}].
\end{rem}

The following remark gives a relationship between the quiver of type $\mathbb{A}_{\infty}$ and the quivers associated to the triangulations on the infinity-gon.

\begin{rem}
If $R$ is a quiver of type $\mathbb{A}_{n}$, then there exists a triangulation $\Gamma$ of the $(n+3)$-gon  such that $R\cong Q_{\Gamma};$ but if we reproduce
the assumption above with $\mathbb{A}_{\infty}$ playing the role of $\mathbb{A}_{n}$ and the infinity-gon playing the role of $(n+3)$-gon, then this is not true.
For example let $Q$ be a linearly oriented quiver of type $\mathbb{A}_{\infty}$ and assume that there exists a triangulation $T$ of the infinity-gon such that $Q\cong Q_{T}$.
Hence $Q$ is a triangulation with left fountain and without right fountain or $Q$ is a triangulation with right fountain and without left fountain. This is a contradiction with
[\ref{holjor}].

\end{rem}

       An isomorphism $f$ between two $\mathbb{Z}$-algebras is called a $strong$ $isomorphism$ of clusters algebras if $f$ maps each cluster to a cluster and preserves
mutations. For more details
we refer to [\ref{fz1}].

\begin{expl}
We let $T= \{(-n,n),(-n-1,-n)/ n\geq 1\}$ and\\ $T'= \{(-n,0),(0,n)/ n\geq 2\},$  be two triangulations of $\mathbb{S}.$
The quiver associated to the triangulation $T$ is the quiver $Q_{T}$ given by:

  $$
\xymatrix @=20pt
{
&1\ar@{->}[r]
&2\ar@{<-}[r]
&3\ar@{->}[r]
&...\ar@{<-}[r]
&n-1\ar@{->}[r]
& n\ar@{<-}[r]
&...
}
$$
and the quiver associated to the triangulation $T'$ is the non connected quiver $Q_{T'}$ given by:

$
\xymatrix @=20pt
{
&2\ar@{->}[r]
&4\ar@{->}[r]
&6\ar@{->}[r]
&...\ar@{->}[r]
&2n\ar@{->}[r]
& 2n+2\ar@{->}[r]
&...
}
$

$
\xymatrix @=20pt
{
&1\ar@{<-}[r]
&3\ar@{<-}[r]
&5\ar@{<-}[r]
&...\ar@{<-}[r]
&2n-1\ar@{<-}[r]
& 2n+3\ar@{<-}[r]
&...
}
.$

There is no strong isomorphism between the cluster algebras $\mathcal{A}(T)$ and $\mathcal{A}(T').$
This shows that the theorem of Fomin, Shapiro and Thursten mentioned above does not hold for the infinity-gon.
\end{expl}

\subsection{The specificity of cluster algebras arising from the infinity-gon}

In this section we find a criterion that allows us to decide whether two triangulations give rise to isomorphic cluster algebras.

\begin{definition}
Two triangulations $T$ and $T'$ are said to be $congruent$ if there exists a bijection
$\theta:\overline{\mathbb{S}}\longrightarrow \overline{\mathbb{S}}$ which maps $T$ to $T'$ and preserves the flips of arcs; that is
$\theta(T)=T'$ and $\theta(f_\gamma)=f_{\theta(\gamma)}$, where $f_\gamma$ is the flip of $\gamma.$

\end{definition}
The bijection $\theta$ is called an $admissible$ map.
If  $T$ and $T'$ are $congruent$, we denote this situation by $T\simeq T'.$ Congruence is an equivalence relation.

\begin{definition}
A sequence of arcs $(\gamma_{n})_{n\geq 1}$ in $\overline{\mathbb{S}}$, is said to
span $\overline{\mathbb{S}}$ if for any arc $\gamma$ there exists an integer $k$ such that $\gamma$ is spanned by $\gamma_{k}.$

\end{definition}

\begin{lem}

Let $T$ be a triangulation of type $(I)$, then there exist a zigzag triangulation $Z$ and a sequence of common arcs
$(\gamma_{n})_{n\geq 1}$ in $T$ and $Z$
 that spans $\overline{\mathbb{S}}$.

\end{lem}
\begin{proof}

Let $T$ be a triangulation of type $(I)$, we shall show that there exists a sequence of distinct arcs $(\gamma_{k_{n}})_{n\geq 1}$ of $T$
such that $\gamma_{k_{n}}=(s_{k_{n}},s_{k_{n}})$, where $s_{k_{n+1}}<s_{k_{n}<0}$ and  $t_{k_{n+1}}>t_{k_{n}>0}.$

     We assume first that any arc $\gamma=(s,t)$, with with $s<0$ and $t>0$ does not belong to $T.$

     Because $(s,t)$ does not belong to $T$ and $T$ is a triangulation, there is an arc $\gamma_{1}=(s_{1},t_{1})$ of $T$
      which crosses
$(s,t)$ and  $\gamma_{1}$ is closer to a vertex $t$ and $t_{1}>t$. Analogously, there exists an arc $\gamma_{2}=(s_{2},t_{2})$ of
de $T$ which crosses
$(s,t)$ and  $\gamma_{2}$ is closer to a vertex $s$ and $s_{2}<s$. The connected components $C_{\gamma_{1}}$ and $C_{\gamma_{2}}$
respectively of $\gamma_{1}$ and $\gamma_{2}$ are distinct, this is a contradiction because $T$ is of type $(I).$ Hence $\gamma\in T.$

          We show that $T$ has an infinite sequence of arcs $(\gamma_{k_{n}})_{n\geq 1}$ such that $\gamma_{k_{n}}=(s_{k_{n}},s_{k_{n}})$, where $s_{k_{n+1}}<s_{k_{n}<0}$
and $t_{k_{n+1}}>t_{k_{n}>0}.$
Assume that any sequence of arcs $(\gamma_{k_{n}})_{n\geq 1}$ such that $\gamma_{k_{n}}=(s_{k_{n}},s_{k_{n}})$, where $s_{k_{n+1}}<s_{k_{n}<0}$
and $t_{k_{n+1}}>t_{k_{n}>0}$ is finite. Let $\gamma_{k_{l}}=(s_{k_{l}},t_{k_{l}})$,
 with $s_{k_{l}}<0$ and $t_{k_{l}}>0$ be an arc of $T$ such for all arcs of $(\gamma_{k_{n}})_{n\geq 1}$, we have
$s_{k_{l}}<s_{k_{n}}$ and  $t_{k_{l}}>t_{k_{n}}.$ The same argument used for $\gamma$ to $\gamma_{k_{l}}$ gives rise to a contradiction.
Thus there exists a sequence of arcs $(\gamma_{k_{n}})_{n\geq 1}$ of $T$ which can be chosen such that
$s_{k_{n+1}}<s_{k_{n}<0}$ and  $t_{k_{n+1}}>t_{k_{n}>0}.$
         Now we construct a zigzag triangulation having infinitely many common arcs with $T$.
Because $(\gamma_{k_{n}})_{n\geq 1}$ is a sequence of infinitely many non-crossing arcs, we complete it in each polygon bounded by $\gamma_{k_{n}}$
and $\gamma_{k_{n+1}}$. We choose one zigzag triangulation of the polygon bounded by $\gamma_{k_{n}}$
and $\gamma_{k_{n+1}}$. We apply this process for all non-negative integers $n$, thus for all polygons bounded by $\gamma_{k_{n}}$
and $\gamma_{k_{n+1}}$. This process gives rise to two subsets of $\mathbb{S}$ as follows.

    The first is composed by the arcs $\gamma_{k_{n}}$, for all $n$.

    The second is the union of the new arcs of the zigzag triangulations of each polygon bounded by $\gamma_{k_{n}}$
and $\gamma_{k_{n+1}}$.

      Let $Z$ be the union of the two subsets defined above. In fact, $Z$ is a triangulation of $S$ by the construction, and each $\gamma_{k_{l}}$ is a common
arc of $T$ and $Z$.

         Finally, we show that $(\gamma_{k_{n}})_{n\geq 1}$ spans $\overline{\mathbb{S}}$.
Let $\delta=(u,v)$ be an arc of $\mathbb{S},$ since $(\gamma_{k_{n}})_{n\geq 1}$ is infinite, there is an integer $l$ such that the arc
 $\delta$ is spanned by $\gamma_{k_{l}}$. This completes the proof.
\end{proof}

\begin{lem}
Let $T$ and $T'$ be two triangulations of type $(II)$ or $T$ and $T'$ be two triangulations of type $(III)_{k}$, and let $C_{T}$ and $C_{T'}$
their connected components respectively. Then there is a sequence of common arcs
$(\gamma_{n})_{n\geq 1}$ in $T$ and $T'$
 that spans $C_{T}$ and $C_{T'}$.
\end{lem}

 \begin{proof}

    $(i)$  Let $T$ and $T'$ be two triangulations of type $(II)$, suppose that $T$ has a left-fountain $m_0$ and a right-fountain $n_0$,
$n_0-m_0=1$. There is a triangulation $\Gamma$ with one fountain such that its associated quiver $Q_{\Gamma}$ is isomorphic to
the associated
quiver $Q_T$ of $T$.

              Because of the above argument, it is sufficient to give a proof just for the case where each triangulation has a left-fountain and a right-fountain.
Let $T$ and $T'$ be two triangulations of type $(II)$ such that $T$ has a left-fountain $m_0$ and a right-fountain $n_0$ and
$T'$ has a left-fountain $m'_0$ and a right-fountain $n'_0$. Because $n_0$ and $n'_0$ are integers, we can assume that
$n_0\leq n'_0;$ and there is a non-negative integer $l$ such that $l=n'_0-n_0.$  We define the map
$\sigma:\overline{\mathbb{S}}\longrightarrow \overline{\mathbb{S}}$ by $\sigma(m,n)=(m+l,n+l)$. The map $\sigma$ is a bijection and
the image $\sigma(T)$ of $T$ is a triangulation. Moreover, $\sigma$ preserves the flips of arcs. Thus $\sigma$ is an admissible map.
Therefore, we can suppose without loss of generality that $T$ and $T'$ have the same left-fountain $m_0$ and the same
right-fountain $n_0$. Because the triangulations  $T$ and $T'$ have the same left-fountain $m_0$ and the same
right-fountain $n_0$, then $T$ and $T'$ have infinitely many common arcs of the form $(n_0,n)$ and so, $T$ and $T'$ have infinitely many common arcs of the form$(m,m_0)$.
Thus there is a sequence of
distinct common arcs $(\gamma_{k_{n}})_{n\geq 1}$ which spans $C_{T}$ and $C_{T'}.$

         $(ii)$   Now we suppose that $T$ and $T'$ are two triangulations of type $(III)_{k}$. By using the argument above, we can assume
without loss of generality that $T$ and $T'$ have the same left-fountain $m_0$ and the same right-fountain $n_0$.

            If $k$ is equal to zero, the proof is similar to the case $(i).$  If $k\geq 1$, the frozen arc $\omega$ bounds a
polygon $\mathbb{P}_\omega$;
and by the definition of $k$, each arc of the restriction of $T$ to $\mathbb{P}_\omega$ is spanned by $\omega$ and each arc of
the restriction of $T'$ to $\mathbb{P}_\omega$ is spanned by $\omega$. Combining this argument and the one used in $(i)$ to define
the common arcs, we find that there exists a sequence of distinct arcs $(\gamma_{k_{n}})_{n\geq 1}$ which spans $C_{T}$ and $C_{T'}.$

\end{proof}

\begin{prop}
Let $T$ and $T'$ be two triangulations of $\mathbb{S},$ then  $T$ and $T'$ are congruent if and only if  $T$ and $T'$ are
of the same type.
\end{prop}

\begin{proof}

Let $T$ and $T'$ be two triangulations of $\mathbb{S},$ and assume that $T\simeq T'.$
Since $T\simeq T',$ there exists an admissible map $\theta$ which maps $T$ to $T'.$
Let $C_{\gamma}$ be a connected component of $T,$ then $\theta(C_{\gamma})$ is a connected component of $T'.$
Because $\theta(\gamma)\in \theta(C_{\gamma})$, we have $\theta(C_{\gamma})=C_{\theta(\gamma)}.$ If $C_{\gamma}$ and
$C_{\delta}$ are two distinct connected components of $T$, then $\theta(C_{\gamma})$ and $\theta(C_{\delta})$ are distinct
connected components of $T'$.  The number of connected components is invariant by $\theta.$ Moreover the connected components
$C_{\gamma}$ and $C_{\theta(\gamma)}$ are both either finite, or infinite. Moreover, if $T$  has a finite component
$C_{\gamma}$, then $|C_{\gamma}|=|C_{\theta(\gamma)}|=k.$
Hence $T$ and $T'$ are of the same type.

          Conversely let $T$ and $T'$ be two triangulations of $\mathbb{S},$ we have three cases.

           $(i)$ If $T$ and $T'$ are of type $(I),$ then by Lemma 3.3 there exists a zigzag triangulation $Z$
and a sequence of common distinct arcs $(\gamma_{k_{n}})_{n\geq 1}$ in $T$ and $Z$
 that spans $\overline{\mathbb{S}}$.

  Let $\mathbb{P}_{k_{n}}$ be the polygon bounded by the arc $\gamma_{k_{n}}.$ The restriction $T_{k_{n}}$
and $Z_{k_{n}}$ of the triangulations $T$ and $Z$ are the triangulations of $\mathbb{P}_{k_{n}}.$ We want to show that
$T=\bigcup\limits_{n\geq 1}^{} T_{k_{n}}$ and $Z=\bigcup\limits_{n\geq 1}^{} Z_{k_{n}}$. Because $T_{k_{n}}\subset T$,
then  $\bigcup\limits_{n\geq 1}^{} T_{k_{n}}\subset T$. For now let us show that $\bigcup\limits_{n\geq 1}^{} T_{k_{n}}$
is a triangulation.
Assume that $\bigcup\limits_{n\geq 1}^{} T_{k_{n}}$ is not a triangulation. We have $\bigcup\limits_{n\geq 1}^{} T_{k_{n}}\subset T$
 and $T$ is a maximal set of arc, then there exists an arc $\varsigma\in T$
which does not belong to $\bigcup\limits_{n\geq 1}^{} T_{k_{n}}.$ Because the sequence $(\gamma_{k_{n}})_{n\geq 1}$
spans $\overline{\mathbb{S}}$, there exists an integer $l$ such that the arc $\varsigma$ is spanned by $\gamma_{k_{l}}.$
Hence $\varsigma$ is an arc of $T_{k_{l}}$ this is a contradiction. Thus $T=\bigcup\limits_{n\geq 1}^{} T_{k_{n}}$.
Analogously, one can show that $Z=\bigcup\limits_{n\geq 1}^{} Z_{k_{n}}.$

     We know that  $T_{k_{n}}$ and  $Z_{k_{n}}$ are related by a sequence of flips. This sequence of flips induces an admissible
map $\theta_n$ which maps $T_{k_{n}}$ to  $Z_{k_{n}}.$ Let $\overline{\mathbb{P}}_{k_{n}}$ be the set of all arcs of
$\mathbb{P}_{k_{n}},$ we have $\overline{\mathbb{P}}_{k_{n}}\subset \overline{\mathbb{P}}_{k_{n+1}}$. We have also
$\bigcup\limits_{n\geq 1}^{}\overline{\mathbb{P}}_{k_{n}}=\overline{\mathbb{S}}$.

        We define the map $\theta:\overline{\mathbb{S}}\longrightarrow \overline{\mathbb{S}}$ by the following: let $\gamma$
be an element of $\overline{\mathbb{S}}$, then there is a minimal integer $n$ such that $\gamma \in \overline{\mathbb{P}}_{k_{n}},$
we set $\theta(\gamma)=\theta_{n}(\gamma).$ By construction, $\theta$ is a bijection which maps $T$ to $Z$ and preserves
the flips of arcs. Hence $T\simeq Z.$

     We reproduce the same reasoning above with $T'$ playing the role of $T$ and $Z'$ playing the role of $Z.$

If $(m_0,n_0)$ is the arc of $Z$  and  $(m'_0,n'_0)$ is the arc of $Z'$ such that $n_0-m_0=2=n'_0-m'_0$. We set $l=n'_0-n_0$
and define the map $\sigma:\overline{\mathbb{S}}\longrightarrow \overline{\mathbb{S}}$ by $\sigma(m,n)=(m+l,n+l).$
Then $\sigma$ is an admissible map which maps $Z$ to $Z'$. Then $Z\simeq Z'$ and thus $T\simeq T'$

          $(ii)$ If $T$ and $T'$ are of the type $(II),$ because $C_T=C_{T'}$ we use Lemma 3.4 and we have a bijection
$\theta$ in $C_T$ which maps $T$ to $T'$ and preserves the flips of arcs. We define $\theta$ on each unreachable arc $\gamma$ by
$\theta(\gamma)=\gamma.$ Thus we have extended the bijection $\theta$ to $\overline{\mathbb{S}}$. Hence $\theta$ is an
admissible map, thus $T\simeq T'.$

           $(iii)$ If $T$ and $T'$ are of the type $(III_{k}),$ the restrictions of the triangulations $T$ and $T'$to the polygon bounded
by the frozen arc are congruent. Because $C_T=C_{T'}$, by using the same principle as in $(b)$, we construct an admissible map $\theta$
which maps $T$ to $T'.$

\end{proof}

\begin{cor}

Let $T$ and $T'$ be two triangulations of $\mathbb{S}$, if $T$ and $T'$ are mutation equivalent, then $T$ and $T'$
are congruent.

\end{cor}
\begin{proof}
Let $T$ and $T'$ be two triangulations of $\mathbb{S}$. Assume that $T$ and $T'$ are flip equivalent. Because the triangulations $T$ and $T'$ are flip equivalent, they are of the same type.
By the proposition 3.2 we have $T\simeq T'.$
\end{proof}

Now we are in position to prove
our second main theorem.

\begin{theo}
Let $T$ and $T'$ be two triangulations of $\mathbb{S},$ $\mathcal{A}(T)$ and $\mathcal{A}(T')$ the associated cluster algebras.
Then $T$ and $T'$ are congruent if and only if the clusters algebras $\mathcal{A}(T)$ and $\mathcal{A}(T')$ are strongly isomorphic.

\end{theo}

\begin{proof}

Let $T$ and $T'$ be two triangulations of $\mathbb{S},$ assume that $T\simeq T'.$ Let $\mathcal{A}(T)$ be a cluster
algebra of seed $(X_{T},Q_{T}):=\sum_{T}$, where $Q_{T}$ is a quiver of $T$ and
$X_{T}=\{x_{\gamma}| \gamma \in T\}$ the set of undeterminates. Let $\mathcal{A}(T')$ be a cluster
algebra of seed $(X_{T},Q_{T'}):=\sum_{T}$ where $Q_{T'}$ is a quiver of $T'$ and $X_{T'}=\{u_{\lambda}| \lambda \in T\}$ the set of undeterminates.

          Since $T\simeq T'$, there exists an admissible bijection $\theta$, such that $\theta(T)=T'.$ We denote by $f_\gamma$ the flip of the arc $\gamma.$
We define $\varphi_{\theta}:\mathcal{A}(T)\longrightarrow \mathcal{A}(T')$ on the cluster variable by $\varphi_{\theta}(x_{\gamma})=u_{\theta(\gamma)}.$
We have on one hand
$$\begin{array}{lllll}\varphi_{\theta}(x_{T})&=&\{\varphi_{\theta}(x_{\gamma})| \gamma \in T\}\\ &=&\{u_{\theta(\gamma)}| \gamma \in T\}\\
&=& \{u_{\theta(\gamma)}| \theta(\gamma) \in \theta(T)\}\\
&=&\{u_{\lambda}| \lambda \in T'\}. \end{array}$$

On the other hand, we have
$$\begin{array}{lllll}\varphi_{\theta}(\mu_{x_{\gamma}})&=&\varphi_{\theta}(x_{f(\gamma)})\\ &=&u_{\theta(f(\gamma)}\\
&=& u_{f(\theta(\gamma))}\\
&=&\mu_{u_{\theta(\gamma)}}(u_{\theta(\gamma)}). \end{array}$$

 We extend $\varphi_{\theta}$ to an isomorphism of $\mathbb{Z}$-algebras from $\mathcal{A}(T)$ to $\mathcal{A}(T')$.
Therefore according to [\ref{asdush}], $\mathcal{A}(T)$ and $\mathcal{A}(T')$ are strongly isomorphic.

           Conversely assume that $\mathcal{A}(T)$ and $\mathcal{A}(T')$ are strongly isomorphic and that $T$ and $T'$ are not congruent. Then
there exists a strong isomorphism $\psi:\mathcal{A}(T)\longrightarrow \mathcal{A}(T').$
According to Proposition 3.2, $T$ and $T'$ are not of the same type. We have four cases to consider.

    $(a)$ $T$ is of type $(I)$ and $T'$ is of type $(II).$
Because $T'$ is of type $(II),$ it has two disjoint connected components. Let $u_{\lambda_{1}}$ and  $u_{\lambda_{2}}$ be two cluster variables such that
$\lambda_{1}$ and $\lambda_{2}$ do not belong to the same connected component. Since $\psi$ is a strong isomorphism, there are two arcs $\gamma_{1}$ and $\gamma_{2}$
such that $\psi(x_{\gamma_{1}})=u_{\lambda_{1}}$ and $\psi(x_{\gamma_{2}})=u_{\lambda_{2}}.$ The two arcs $\gamma_{1}$ and $\gamma_{2}$ are related by a sequence
of flips and then the two variables $x_{\gamma_{1}}$ and $x_{\gamma_{2}}$ are related by a sequence of mutations, because $T$ is of type $(I).$
In fact, $\psi$ is a strong isomorphism, then $u_{\lambda_{1}}$ and  $u_{\lambda_{2}}$ are related to a sequence of mutations. Thus $\lambda_{1}$ and $\lambda_{2}$ are
related by a sequence of flips. This is a contradiction, because $\lambda_{1}$ and $\lambda_{2}$ do not belong to the same connected component.

        $(b)$ $T$ is of type $(I)$ and $T'$ is of type $(III)_{k}.$
The proof is analogous of the one in the case $(a).$

         $(c)$ $T$ is of type $(II)$ and $T'$ is of type $(III)_{k}.$
The triangulation $T'$ has a frozen arc $\omega$, then the cluster algebra $\mathcal{A}(T')$ has a frozen cluster variable $u_{\omega}$ in the sense of [\ref{asdush}].
The map $\psi$ is an isomorphism, then it maps a frozen variable $x_{\gamma}$ of $\mathcal{A}(T)$ to a cluster variable  $u_{\omega}=\psi(x_{\gamma}).$ Because $\psi$ is
a strong isomorphism, it maps frozen variable to frozen variable while $\mathcal{A}(T)$ has no
frozen variable. This is a contradiction.

        $(d)$ $T$ is of type $(III)_{k}$ and $T'$ is of type $(III)_{k'}$ with $k\neq k'.$
We assume without loss of generality that $k>k'.$ Let $\omega$ and $\omega'$ be respectively the frozen arcs of $T$ and $T'.$ Because $\psi$ is a strong isomorphism , we have
$\psi(x_{\omega})=u_{\omega'}.$  Since $k>k',$ there is an arc $\gamma$ spanned by $\omega$ such that $\psi(x_{\gamma})=u_{\lambda},$ and $\lambda$ is not spanned by $\omega'.$
The connected component $C_{\gamma}$ is finite, because $\gamma$ is spanned by $\omega$ and the connected component $C_{\lambda}=C\psi(_{\gamma})$ is infinite.
This is a contradiction, because $\psi$ is a strong isomorphism.

\end{proof}

 Now we want to show that each cluster algebras of type $\mathbb{A}_{\infty}$ can be embedded in a cluster algebra arising from $\mathbb{S}$.

\begin{lem}
Let $Q$ be a quiver mutation equivalent to a quiver of type $\mathbb{A}_{\infty}$. Then $Q$ is not a quiver associated to a triangulation of $\mathbb{S}$ if and only if $Q$ has a
subquiver of type $\mathbb{A}_{\infty}$ with linear orientation.

\end{lem}
\begin{proof}
 Assume that $Q$ has a subquiver of type $\mathbb{A}_{\infty}$ with orientation not necessarily linear. $Q$ is a connected quiver, Q has a subquiver of type $\mathbb{A}_{\infty}$
 with orientation not necessarily linear and $Q$ is mutation equivalent to a quiver of type  $\mathbb{A}_{\infty}.$ There is a triangulation $T$ of $\mathbb{S}$ such that
 $Q_T\cong Q$.

               Conversely, assume that $Q$ has a subquiver of type $\mathbb{A}_{\infty}$ with linear orientation. It is sufficient to show that the quiver $R$: $1\longrightarrow
 2\longrightarrow ...$
is not the quiver associated to any triangulation. Suppose that there exists a triangulation $T$ such that $Q_T=R$. We denote by $\tau_i$ the arc of $T$ corresponding to the vertex $i$.
All $\tau_i$ , where $i$ is a non-negative integer, have the same origin and are the arcs of the same half-line. $T=\{\tau_{i}| i\geq 1\}$ is a triangulation of $\mathbb{S}$ with left-fountain,
but no right-fountain. This is a contradiction see [\ref{holjor}].
\end{proof}

\begin{cor}
Let $Q$ be a quiver mutation equivalent to a quiver of type $\mathbb{A}_{\infty}$. Let $\underline{u}=\{u_{i}| i\geq 1\}$ the set of undeterminates attached to the vertices of $Q.$ Then
there exists a seed
$\Sigma_T=(\underline{x}_T,Q_T)$ associated to a triangulation $T$ of $\mathbb{S}$ and a cluster embedding $\eta:\mathcal{A}(\underline{u},Q)\hookrightarrow \mathcal{A}(\Sigma_T).$

\end{cor}

\begin{proof}
Assume that $Q$ is a  mutation equivalent to a quiver of type $\mathbb{A}_{\infty}$. If $Q$ has a subquiver with orientation not necessarily linear, then  $Q$ is the quiver associated to a
triangulation
of $\mathbb{S}$; in this case $\eta$ is the identity morphism. Hence the result.

        Assume now that $Q$ has a subquiver of type $\mathbb{A}{\infty}$ with linear orientation. By Lemma 3.5, $Q$ is not the quiver of any triangulation of $\mathbb{S}$. We define the quiver $R$
of type $\mathbb{A}_{\infty}$ with linear orientation distinct to the one of $Q.$
The quiver $Q\cup R$ is the quiver associated to a triangulation $T$ of $\mathbb{S}.$ Since the quivers $Q$ and $R$ are non-intersecting connected quivers, the inclusion of the quiver $\eta_0:Q\subset Q\cup R$ induces an embedding of cluster algebras $\eta:\mathcal{A}
(\underline{u},Q)\hookrightarrow \mathcal{A}(\Sigma_T).$ Because the corestriction $\eta_0:Q\longrightarrow \eta_0(Q)=Q$ is the identity map, the $\mathbb{Z}$-morphism $\eta$ is a cluster morphism.
\end{proof}

\section{The cluster category of associated to $\mathbb{S}$}

 \subsection{The infinite cluster category  of type $A_{\infty}$}

We recall the description of the infinite cluster category given in [\ref{pj}, \ref{holjor}].
Let $K$ be a field and $R=K[T]$ be the polynomial algebra. We view $R$ as a differential graded algebra with zero differential and $T$ placed
in homological
 degree 1. Then we set
$D^{f}(R)$ be the derived category of differential graded $R$-modules with finite dimensional homology over $K$, then $\mathcal{D}=D^{f}(R)$
is the infinite cluster category of type $\mathbb{A}_{\infty}$. The suspension and the
Serre functor of $\mathcal{D}$ are denoted by $\Sigma$ and $S$ respectively.
The category $\mathcal{D}$ is a $K$-linear, Hom-finite, Krull-Schmidt,
triangulated and 2-Calabi-Yau category whose Auslander-Reiten quiver is of the form $\mathbb{Z}\mathbb{A}_{\infty}$, we refer to [\ref{pj}].
The Auslander-Reiten translation
of $\mathcal{D}$ is $\tau=S\Sigma^{-1}=\Sigma.$
For a given integer $r\geq 0$,
 we have
a differential graded $R$-module $X_{r}=R/(T^{r+1})$ which is concentrated in homological degrees from 0 to $r$. The indecomposable objects
 of $\mathcal{D}$ are
$\Sigma^{j}X_{r}$ for $j,r$ integers, $r\geq 0$ and $\Sigma$ the shift of $\mathcal{D}$. The Auslander-Reiten quiver $\Gamma(\mathcal{D})$
of $\mathcal{D}$ is of the form

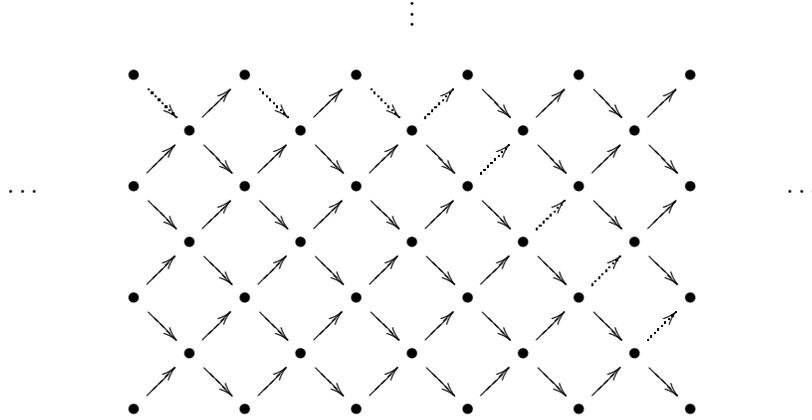
\[\xymatrix@!@-1.00pc@R-1.00pc@C-1.00pc{
&&&&&&& \vdots &&&& \\
&& \bullet \ar@{.>}[dr]  & &  \bullet \ar@{.>}[dr] & &  \bullet \ar@{.>}[dr] & &   \bullet \ar[dr] & &   \bullet \ar[dr] & &   \bullet  \\
 && &  \bullet  \ar[dr] \ar[ur]  & &  \bullet \ar[dr] \ar[ur] & &   \bullet  \ar[dr] \ar@{.>}[ur] & &   \bullet  \ar[dr]\ar[ur]  & &
 \bullet \ar[dr] \ar[ur] &&&&&&&&&& \\
\ldots && \bullet \ar[dr] \ar[ur] & & \bullet  \ar[dr] \ar[ur]  & &   \bullet \ar[dr] \ar[ur] & &   \bullet \ar[dr] \ar@{.>}[ur] & &
\bullet  \ar[dr]\ar[ur] & &  \bullet && \ldots  \\
&&&  \bullet \ar[dr] \ar[ur] & &  \bullet \ar[dr] \ar[ur] & &   \bullet \ar[dr] \ar[ur] & &   \bullet \ar[dr] \ar@{.>}[ur] & &
\bullet \ar[dr] \ar[ur] \\
&& \bullet \ar[ur] \ar[dr]  & &  \bullet \ar[ur] \ar[dr] & &  \bullet \ar[ur] \ar[dr] & &   \bullet \ar[ur] \ar[dr]& &   \bullet
\ar@{.>}[ur] \ar[dr] & &   \bullet  \\
&&&  \bullet  \ar[ur] \ar[dr] & &  \bullet  \ar[ur] \ar[dr]  & &   \bullet  \ar[ur] \ar[dr] & &   \bullet  \ar[ur] \ar[dr]  & &
\bullet  \ar@{.>}[ur] \ar[dr]  \\
&&  \bullet  \ar[ur]  & &  \bullet  \ar[ur] & &  \bullet \ar[ur] & &  \bullet  \ar[ur] & &   \bullet  \ar[ur]  & &   \bullet \\
  }\]

By using the identification $(m,n):=\Sigma^{-n}X_{n-m-2}$, we have
the following representation of the quiver $\Gamma(\mathcal{D})$
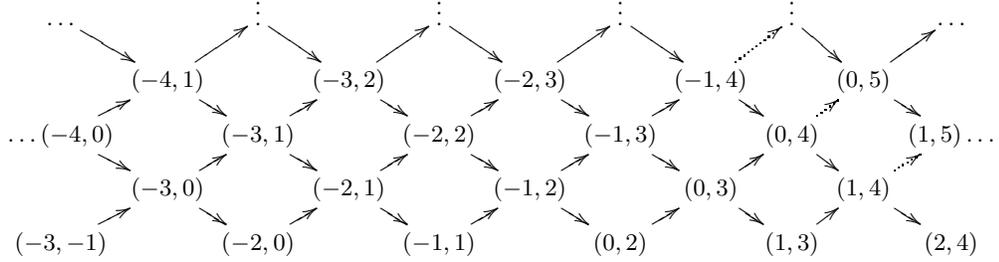
\begin{center}
{\small
\[\xymatrix@R=6pt@C=1pt{
&&&&&&& &&&& \\
&& \ldots \ar[dr] & & \vdots \ar[dr] & & \vdots \ar[dr] & & \vdots \ar[dr] & &
\vdots \ar[dr] & & \ldots && \\
&&& (-4,1) \ar[dr] \ar[ur] & & (-3,2) \ar[dr] \ar[ur] & & (-2,3) \ar[dr] \ar[ur] & & (-1,4) \ar[dr] \ar@{.>}[ur] & &
(0,5) \ar[dr] \ar[ur] \\
&& \ldots (-4,0) \ar[ur] \ar[dr] & & (-3,1) \ar[ur] \ar[dr] & & (-2,2) \ar[ur] \ar[dr] & & (-1,3) \ar[ur] \ar[dr]& & (0,4)
\ar@{.>}[ur] \ar[dr] & & (1,5) \ldots \\
&&& (-3,0) \ar[ur] \ar[dr] & & (-2,1) \ar[ur] \ar[dr] & & (-1,2) \ar[ur] \ar[dr] & & (0,3) \ar[ur] \ar[dr] & &
(1,4) \ar@{.>}[ur] \ar[dr] \\
&& (-3,-1) \ar[ur] & & (-2,0) \ar[ur] & & (-1,1) \ar[ur] & & (0,2) \ar[ur] & & (1,3) \ar[ur] & & (2,4) \\
}\]
}

\end{center}

        The identification of the indecomposable objects $\Sigma^{-n}X_{n-m-2}$ of $\mathcal{D}$ given by $(m,n):=\Sigma^{-n}X_{n-m-2}$
is called the $standard$ $coordinates$ system on $\Gamma(\mathcal{D})$.
         The morphisms between indecomposable objects are described as follows:
Let $x=\Sigma^{-j}X_{j-i-2}$ be a vertex of the Auslander-Reiten quiver of $\mathcal{D}$, we define the sets $H^{-}(x)$ and $H^{+}(x)$
of  vertices of the Auslander-Reiten quiver as $$H^{-}(x)=\{\Sigma^{-n}X_{n-m-2}| m\leq i-1, i+1\geq n \leq j-1\}$$ and
$$H^{+}(x)=\{\Sigma^{-n}X_{n-m-2}|  i+1\geq m \leq j-1, j+1\leq n\}.$$
 We write $H(x)=H^{-}(x)\cup H^{+}(x)$. This situation can be sketched as follows.

 
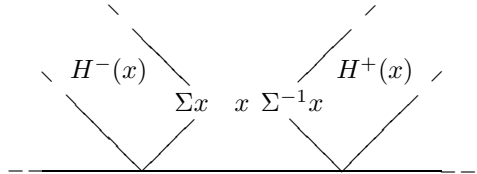
\[
\xymatrix @-4.5pc @! {
    &&&*{} &&&&&&&& *{}&& \\
    &&&& *{} \ar@{--}[ul] & & & & & & *{} \ar@{--}[ur] \\
    &*{}&& H^-(x) & & & & & & & & H^+(x) && *{}\\
    &&*{}\ar@{--}[ul]& & & & {\Sigma x \hspace{3ex}} \ar@{-}[ddll] \ar@{-}[uull] & {x} & {\hspace{3ex} \Sigma^{-1} x} \ar@{-}[ddrr] \ar@{-}[uurr]&
     & &&*{}\ar@{--}[ur]&\\
    && \\
    *{}\ar@{--}[r]&*{} \ar@{-}[rrr] && &  *{} \ar@{-}[uull] \ar@{-}[rrrrrr]& & & & & &  *{} \ar@{-}[uurr]\ar@{-}[rrr]&&&*{}\ar@{--}[r]&*{}\\
           }
\]
Moreover we have

 $H^{-}(x)=\{\Sigma^{-n}X_{n-m-2}/ m\leq i-1, i+1\geq n \leq j-1\}$ and $H^{+}(x)=\{\Sigma^{-n}X_{n-m-2}/  i+1\geq m \leq j-1, j+1\leq n\}$.
 We write $H(x)=H^{-}(x)\cup H^{+}(x)$.

          The following proposition in [\ref{holjor}] characterizes the morphisms of $\mathcal{D}.$
\begin{prop}
Let $x$ and $y$ be two indecomposable objects of $\mathcal{D}$. Then

$\begin{array}{cccc}
{}& Hom_{\mathcal{D}}(x,y)&= &\left\{
\begin{array}{lll} K\quad \texttt{if}\quad{ y\in H(\Sigma x)}
\\ 0\quad\quad \texttt{if not}\end{array}\right.\end{array}$

\end{prop}
The following description is due to holm and J{\o}rgensen [\ref{holjor}, Remark 2.4]
\begin{rem}
There are two distinct types of non-zero morphisms going from $x$ to indecomposable objects of $\mathcal{D}$: those going to objects in $H^{+}(x)$
are called $forward$ $morphisms$, and those going to objects $H^{-}(x)$
are called $backward$ $morphisms$.

   The forward morphisms have an easy model: up to multiplication by a nonzero
scalar, they are induced by certain canonical morphisms of differential graded modules.
the $backward$ $morphisms$ cannot be seen in the Auslander-Reiten quiver; they are in the infinite radical of $\mathcal{D}$.
\end{rem}

 \subsection{The category of diagonals of the infinity-gon}

 In this section we provide a geometric realization of the category $\mathcal{D}$.

         We adopt the same philosophy as that of [\ref{holjor}], that is, the integers can be viewed as the vertices of the infinity-gon
  and the pairs of integers can be viewed as the arcs of the infinity-gon. Let $(m,n)$ be an arc of infinity-gon, with $m< n$. If $n-m=1,$ we
  say that the arc $(m,n)$ is a boundary arc, and if $m\leq n-2,$ we say that $(m,n)$ is a diagonal of the infinity-gon.
  Our construction is similar to that of [\ref{ccs}] in the case of the $(n+3)$-gon.

       One can define a combinatorial $K$-linear category $\mathcal{C}$ as follows:

        The indecomposable objects are the arcs $(m,n)$ of $\mathbb{S},$ with $m,n\in \mathbb{Z}$ and $m\leq n-2$; the objects of $\mathcal{C}$
  are the linear combinations of the arcs and each arc is stable by the product of the scalars of $K$. The boundary arcs are identified to zero.
  The space of morphisms between two arcs $(m,n)$ and $(p,q)$ of this section is given by:

$\begin{array}{cccc}
{}& Hom_{\mathcal{C}}((m,n),(p,q))&= &\left\{
\begin{array}{lll} K\quad \texttt{if}\quad{ (p,q)\in F_R^{(m,n)}\cup F_L^{(m,n)}}
\\ 0\quad\quad \texttt{if not}\end{array}\right.\end{array}$ \\
where $F_R^{(m,n)}=\{(l,k)\mid m\leq l\leq n-2, k\geq n\}$ and $F_L^{(m,n)}=\{(k,s)\mid
m+2\leq s\leq n, k\leq m \}$. The morphisms between two objects are direct sums of morphisms between arcs.
The composition of morphisms between arcs is given by the product
of scalars in $K.$
The construction of
$\mathcal{C}$ is inspired by the standard coordinates used in  [\ref{holjor}]. The category $\mathcal{C}$ is a category
generated by all the diagonals of $\mathbb{S}$. Therefore by construction $\mathcal{C}$ is $K$-linear, Hom-finite and Krull-Schmidt.
Our main result of this section is the following.
\begin{theo}
The categories $\mathcal{C}$ and $\mathcal{D}$ are equivalent.

\end{theo}

\begin{proof}
Let $F_{0}:ind\mathcal{C}\longrightarrow ind\mathcal{D}$ be such that, for $(m,n)\in ind\mathcal{C}$ we have  $F_{0}(m,n)=
\Sigma^{-n}X_{n-m-2}.$ According to [\ref{holjor}], $F_{0}$ is a bijection. One can define the additive
 functor $F:\mathcal{C}\longrightarrow \mathcal{D}$ as follows:

       $F(m,n)=F_{0}(m,n),$ and we extend $F$ by additivity and $K$-linearity to all objects of $\mathcal{C}.$
 Let $u_{\alpha}:(m,n)\longrightarrow
 (p,q)$ be a morphism of $\mathcal{C}$ which is identified
 with the scalar $\alpha$ of $K$.

        We recall that, via the standard coordinates defined above, if $F(m,n)=x$ and $F(p,q)=y$
  then
 $(p,q)\in F_{R}^{(m,n)}$ if and only if $y\in H^{+}(\Sigma x)$.  We have also  $(p,q)\in F_{L}^{(m,n)}$ if and only if $y\in H^{-}(\Sigma x)$.

         On the one hand, if $(p,q)\in F_{R}^{(m,n)}$ then
  $y\in H^{+}(\Sigma x)$; let $f:x\longrightarrow y$ be a forward morphism of $\mathcal{D}$ that is $f$ is induced by a canonical morphism
         of $DG$-modules. Then each morphism from $x$
 to $y$ is of the form $\lambda f$ where $\lambda\in K$ and we set $F(u_{\alpha})=\alpha f=f_{\alpha}$. On the other hand, if
 $(p,q)\in F_{L}^{(m,n)}$ then
  $y\in H^{-}(\Sigma x)$; let $\bar{g}:x\longrightarrow y$ be a backward morphism, because the category $\mathcal{D}$ is $2$-Calabi-Yau that is
  Hom$_{\mathcal{D}}(x,y)$$=$$D$Hom$_{\mathcal{D}}(y,S(x))$, where $D$$=$Hom$(-,K)$ is the usual duality.
  The morphism $\bar{g}$ is the
  isomorphic image of a forward morphism $g:y\longrightarrow \Sigma^{2}x$.
  We set $F(u_{\alpha})=\alpha \bar{g}=\bar{g}_{\alpha}$ and $F(1_{(m,n)})=1_{x}.$ Let us show now that $F$ is a functor.
  Let  $u_{\alpha}:(m,n)\longrightarrow (p,q)$ and  $u_{\beta}:(p,q)\longrightarrow (r,s)$ where $F(m,n)=x$, $F(p,q)=y$ and $F(r,s)
  =z.$ The proof is completed in three steps $(a)(b)(c)$.

       $(a)$ If $(p,q), (r,s)\in F_{R}^{(m,n)}$ and $(r,s)\in F_{R}^{(p,q)}$ then $y,z\in H^{+}(\Sigma x)$ and $z\in H^{+}(\Sigma y)$.
       We have $F(u_{\alpha})=\alpha f$ and $F(u_{\beta})=
  \beta g$, where $f:x\longrightarrow y$ and $g:y\longrightarrow z$ are forward morphisms. The morphism
  $u_{\beta}u_{\alpha}=u_{\beta \alpha}$ is a morphism from $(m,n)$ to $(p,q)$. Then $F(u_{\beta \alpha})=\beta \alpha h$, where
  $h:x\longrightarrow z$ is a forward morphism of $\mathcal{D}$.
  According to [\ref{holjor}, Lemma 2.5], the morphism $g$ is nonzero and we have the following commutative triangle

$$ \xymatrix{x\ar[rr]^{h}\ar[dr]_{f'}&&z \\
&y\ar[ur]_{g}&} $$
where $f'$ is the morphism induced by a canonical morphism of differential graded modules. By uniqueness of the canonical morphism between
two indecomposables objects, we have $f'=f$ and thus $F(u_{\beta \alpha})=
F(u_{\beta})F({\alpha}).$

        $(b)$ If $(p,q), (r,s)\in F_{g}^{(m,n)}$ and $(r,s)\in F_{d}^{(p,q)}$ then $y,z\in H^{-}(\Sigma x)$ and $z\in H_{+}(\Sigma y).$ We have
$F(u_{\alpha})=\alpha \bar{f}$ and $F(u_{\beta})=\beta g$  where  $g:y\longrightarrow z$ is a morphism induced by a canonical morphism
of differential graded modules. So, $\bar{f}:x\longrightarrow y$ is the isomorphic image of a morphism $f:y\longrightarrow \Sigma^{2}x$ induced by a
canonical morphism of differential graded modules. Since $g$ is a nonzero morphism, in accordance with [\ref{holjor}, Lemma 2.7], we have the
following commutative
 triangle

 $$ \xymatrix{x\ar[rr]^{\bar{h}}\ar[dr]_{\bar{f'}}&&z \\
&y\ar[ur]_{g}&} $$
where $\bar{h}$ is the isomorphic image of a forward morphism $h$ and $\bar{f'}$ is
the image of the morphism $f':y\longrightarrow \Sigma^{2}x$ which is induced by the canonical morphism of $DG$-modules from $y$ to
$\Sigma^{2}x.$  By uniqueness of the canonical morphism between
two indecomposables objects, we have $\bar{f'}=\bar{f}$ and hence $F(u_{\beta \alpha})=F(u_{\beta})F({\alpha}).$ For all other cases not
mentioned above, the
composition of morphisms are equal to zero see [\ref{holjor}, Corollary 2.3]. This shows that $F$ is a functor.

$(c)$ $F$ is essentially surjective because by the definition, each indecomposable module of $\mathcal{D}$ is the image of an arc of
$\mathcal{C}$ under $F.$

           The map $F$:Hom$_{\mathcal{C}}((m,n),(p,q))\longrightarrow $Hom$_{\mathcal{D}}(x,y)$ which associates to $u_{\alpha},$ the
function $F(u_{\alpha})$ is a bijection because of the step $(a)$ and $(b)$. Therefore $F$ is full and faithful.

    Finally, it follows from $(a), (b), (c)$ that $F$ is an equivalence.

\end{proof}
We can give now the description of the category $\mathcal{C}$, via the equivalence established above; clearly, the category $\mathcal{C}$ is
triangulated, $2$-Calabi-Yau and has Auslander-Reiten triangles. In addition, the suspension is given by $(m,n)[1]=(m-1,n-1)$ and the
Serre functor is given by $S(m,n)=(m-2,n-2);$ this situation was predictable from  Holm and Jorgensen in [\ref{holjor}].

     We have also the following operations between the arcs of $\mathbb{S}$ defined by:
${}_{s} \!(m,n)=(m+1,n)$ and $(m,n)_{e}=(m,n+1).$ These operations are defined for the $n+3$-gon in [\ref{ccs}] and for marked surfaces without
punctures
in [\ref{bruzh}]. The operations ${}_{s} \!(m,n)$ and $(m,n)_{e}$ can be extended as functors in the category $\mathcal{C}.$

\begin{prop}
The following statements are equivalent.\\
(a) Hom$((i,j),(p,q)[1])\neq 0$\\
(c)$(i,j)$ and $(p,q)$ cross\\
(d) $(p,q)={}_{s^z} \!(i,j)_{e^{n}} $ or $(p,q)={}_{s^{-n-2}} \!(i,j)_{e^{-r}}$ \\
where $n\geq 0, 0\leq z\leq l-2, 0\leq r\leq l$ and $j-i=l.$

\end{prop}

\begin{proof}
Let $(i,j)$ and $(p,q)$ be two arcs of $\mathcal{C}$.\\ Then we have Ext$^1((i,j),(p,q))=$Hom$((i,j),(p,q)[1])$\\
By  theorem 3.1 and [\ref{holjor}, Lemma 3.5], we have  Hom$((i,j),(p,q)[1])\neq 0$ if and only if the arcs $(i,j)$ and $(p,q)$ cross; that means,
 $(a)$ and
$(c)$ are equivalent.

       Now assume that Hom$((i,j),(p,q)[1])\neq 0$. By the definition of the morphism spaces of $\mathcal{C},$ we have $m\leq p-1\leq n-2$ and
       $n\leq q-1$, or
$m+2\leq q-1\leq n$ and $q-1\leq m$. Because $i$ and $j$ are integers, $l=j-l$ is a positive integer. If we consider the integers $n, z, r$ such
that $n\geq 0$,
$0\leq z\leq l-2$
and $0\leq r\leq l,$ then we have $(p,q)=(i+z,j+n)$ or $(p,q)=(i-n-2,j-r).$ By definition, we have ${}_{s} \!(i,j)=(i+1,j)$ and $\!(i,j)_{e}=(i,j+1),$
thus $(p,q)={}_{s^z} \!(i,j)_{e^{n}} $ or $(p,q)={}_{s^{-n-2}} \!(i,j)_{e^{-r}}$. It follows that $(a)$ and
$(d)$ are equivalent.

\end{proof}

\begin{cor}
Let $(m,n)$ be a diagonal of the infinity-gon, then there is an Auslander-Reiten triangle in $\mathcal{C}$ as
follows
$$(m,n)\longrightarrow {}_s \!(m,n)\oplus (m,n)_e\longrightarrow {}_s \!(m,n)_e\longrightarrow (m,n)[1].$$
Moreover, all Auslander-Reiten triangles of $\mathcal{C}$ are of this form.
\end{cor}

\begin{proof}
It is shown in [\ref{pj}] that the following triangle
$$\Sigma^{-n}X_{u}\longrightarrow \Sigma^{-n}X_{v}\oplus \Sigma^{-n}X_{n-m-1}\longrightarrow \Sigma^{-n-1}
X_{u}\longrightarrow \Sigma^{-n+1}X_{u}$$
is an Auslander-Reiten triangle in $\mathcal{D}$, where $u=n-m-2, v=n-m-3$. By using the equivalence $F$ of
Theorem3.1, we have
$$(m,n)\longrightarrow (m+1,n)\oplus (m,n+1)\longrightarrow (m,n+1)\longrightarrow (m-1,n-1).$$ That is
$$(m,n)\longrightarrow {}_s \!(m,n)\oplus (m,n)_e\longrightarrow {}_s \!(m,n)_e\longrightarrow (m,n)[1].$$
Assuming now that
$$(m,n)\longrightarrow \bigoplus\limits_{i=1}^{n} (m_i,n_i)\longrightarrow (p,q)\longrightarrow (m,n)[1]$$
is an Auslander-Reiten triangle of $\mathcal{C}.$ Since $F$ is an equivalence of categories,
$$F(m,n)\longrightarrow \bigoplus\limits_{i=1}^{n} F(m_i,n_i)\longrightarrow (p,q)\longrightarrow F((m,n)[1])$$
is an Auslander-Reiten triangle of $\mathcal{D};$
that is
$$\Sigma^{-n}X_{n-m-2}\longrightarrow \bigoplus\limits_{i=1}^{l} \Sigma^{-n_i}X_{n_i-m_i-2}\longrightarrow
\Sigma^{-q}X_{q-p-2}\longrightarrow \Sigma^{-n+1}X_{n-m-2}$$ is an Auslander-Reiten triangle of $\mathcal{D}$.
 The form of the Auslander-Reiten triangle
of $\mathcal{D}$ is well known; by identification, we have $\Sigma^{-q}X_{q-p-2}=\Sigma^{-n-1}X_{n-m-2}$. There
exist
$r,s$ with $1\leq r,s\leq l$ such that $F(m_r,n_r)=\Sigma^{-n}X_{n-m-3}, F(m_s,n_s)=\Sigma^{-n}X_{n-m-1}$ and
$F(m_i,n_i)=0,$
for all $i$ different from $r$ and $s.$ This completes the proof of our assertion.

\end{proof}

Let $\mathcal{E}$ be a subcategory of $\mathcal{C}$. Perpendicular subcategories of $\mathcal{E}$ are defined by:
$\mathcal{E}^{\perp}$ is the set of all objects $x\in \mathcal{C}$ such that Hom$_{\mathcal{C}}(b,x)=0$ for $b\in \mathcal{B}\}$ and
${}^{\bot} \!\mathcal{E}$ is the set of all objects $x\in \mathcal{C}$ such that Hom$_{\mathcal{C}}(x,b)=0$ for $b\in \mathcal{B}\}$
A subcategory $\mathcal{H}$ of $\mathcal{C}$ is a weak cluster tilting if it satisfies $(\Sigma^{-1}
\mathcal{H})^{\perp}={}^{\bot} \!(\Sigma\mathcal{H})$.

        For a weak cluster tilting subcategory $\mathcal{H}$ of $\mathcal{C}$ we can consider the set $H$ of
 indecomposable objects of $\mathcal{H}$, whence $ \mathcal{H}$$=$add$H.$ The following corollary is a geometric interpretation of the Theorem 4.3 in [\ref{holjor}].

\begin{cor}
Let $T$ be a set of arcs of $\mathbb{S}$. Then $T$ is a triangulation of $\mathbb{S}$ if and only if $\mathcal{T}$$=$add$T$ is a weak cluster tilting subcategory of $\mathcal{C}.$
\end{cor}

\begin{proof}
Let $T$ be a triangulation of $\mathbb{S}$. Let $U$ be the image of $T$ by the equivalence of Theorem 3.1 and [\ref{holjor}, Theorem 4.3],
then $\mathcal{U}$$=$add$U$ is a weak cluster tilting subcategory of $\mathcal{D}$. Hence $\mathcal{T}$$=$add$T$ is a weak cluster tilting
subcategory of $\mathcal{C}$. Conversely, according to the equivalence of Theorem 3.1, if $\mathcal{U}$ is a cluster tilting subcategory of $\mathcal{C}$, then the set of distinct arcs
$U$ such that $\mathcal{U}$$=$add$U$ is a triangulation of $\mathbb{S}.$

\end{proof}

      \addcontentsline{toc}{chapter}{Bibliographie}

$$   $$

D\'{e}partement de Mathématiques, Universit\'{e} de Sherbrooke
2500,boul. de l'Universit\'{e},
Sherbrooke, Qu\'{e}bec, J1K2R1
Canada
.\\
E-mail address: ndoune.ndoune@usherbrooke.ca
\end{document}